\documentclass{article}[11pt]

\usepackage{amssymb}
\usepackage{amsmath}
\usepackage{amscd}
\usepackage{amsthm}
\usepackage{graphics}
\usepackage{epsfig}
\usepackage{url}
\usepackage{mdwlist}

\usepackage{algorithm}

\usepackage{pgf}
\usepackage{tikz}
\usetikzlibrary{arrows,patterns,plotmarks,shapes,snakes,er,3d,automata,backgrounds,topaths,trees,petri,mindmap}

\usepackage{verbatim} 

\newcommand*{\CopyCounter}[2]{%
  \expandafter\def\csname c@#2\endcsname{\csname c@#1\endcsname}%
  \expandafter\def\csname p@#2\endcsname{\csname p@#1\endcsname}%
  \expandafter\def\csname the#2\endcsname{\csname the#1\endcsname}}
\newcounter{Theorem}
\numberwithin{Theorem}{section}
\CopyCounter{Theorem}{Proposition}
\CopyCounter{Theorem}{ProposedProblem}
\CopyCounter{Theorem}{Property}
\CopyCounter{Theorem}{Claim}
\CopyCounter{Theorem}{Lemma}
\CopyCounter{Theorem}{Corollary}
\CopyCounter{Theorem}{Conjecture}
\CopyCounter{Theorem}{Definition}
\CopyCounter{Theorem}{Example}
\CopyCounter{Theorem}{Remark}
\CopyCounter{Theorem}{Question}
\CopyCounter{Theorem}{Condition}
\CopyCounter{Theorem}{Criterion}
\CopyCounter{Theorem}{Observation}
\theoremstyle{plain}
\newtheorem{thm}[Theorem]{Theorem}

\newtheorem{lemma}[Lemma]{Lemma}

\theoremstyle{definition}

\newtheorem{remark}[Remark]{Remark}

\newif\ifnever\neverfalse

\newcommand{\bm}[1]{\mbox{\boldmath${#1}$}}
\newcommand{\Assume}[1]{{\bf (A#1)}}
\newcommand{\Bssume}[1]{{\bf (B#1)}}
\newcommand{\AssumePrime}[1]{{\bf (A#1')}}

\newcommand{\domain}{\Omega}
\newcommand{\cdomain}{\widebar{\Omega}}
\newcommand{\boundary}{\partial \domain}

\newcommand{\stay}{\M}
\newcommand{\fboundary}{\partial \stay}
\newcommand{\move}{\domain \backslash \stay}

\newcommand{\blds}[1]{\mbox{\scriptsize \boldmath $#1$}}

\newcommand{\imply}{\Longrightarrow}

\newcommand{\xBar}{\bm{\bar{x}}}
\newcommand{\xbar}{\bar{x}}
\newcommand{\ybar}{\bm{\bar{y}}}

\newcommand{\xBtilde}{\bm{\tilde{x}}}
\newcommand{\xtilde}{\tilde{x}}

\newcommand{\yhat}{\bm{\hat{y}}}

\newcommand{\x}{\bm{x}}

\newcommand{\TT}{\mathcal{T}}

\newcommand{\ba}{\bm{a}}

\newcommand{\bb}{\bm{b}}

\newcommand{\bv}{\bm{v}}

\newcommand{\y}{\bm{y}}

\newcommand{\widebar}{\overline}
\newcommand{\R}{\bm{R}}

\newcommand{\J}{{\cal{J}}}

\newcommand{\fB}{\bm{f}}

\newcommand{\M}{{\cal{M}}}
\newcommand{\B}{{\cal{B}}}
\newcommand{\C}{{\cal{C}}}
\newcommand{\D}{{\cal{D}}}
\newcommand{\V}{{\cal{V}}}
\newcommand{\tauP}[1]{\frac{\partial \tau}{\partial \xi_{#1}}}

\DeclareMathOperator*{\argmin}{argmin}

\newcommand{\marginfix}{
\setlength{\parskip}{0.01cm}
\setlength{\textwidth}{6.0in}
\setlength{\oddsidemargin}{-0.0 in}
\setlength{\evensidemargin}{0.0 in}
\setlength{\topmargin}{-0.5in}
\setlength{\textheight}{9.0 in}
}

\marginfix

  \renewenvironment{thebibliography}[1]{%
    \begin{oldthebibliography}{#1}%
      \setlength{\parskip}{.3ex}%
      \setlength{\itemsep}{.3ex}%
  }%
  {%
    \end{oldthebibliography}%
  }

\begin{document}


\vspace*{1cm}
{\Large{\bf
\centerline{Deterministic control of randomly-terminated
processes.}
}}

\vspace*{.1in}
{\Large
\centerline{J. Andrews and A. Vladimirsky\footnote{
\mbox{This research was supported 
by
the National Science Foundation grants DMS-0514487 and DMS-1016150.}
\mbox{The first author's research was also supported
by the NSF Graduate Fellowship.}}}

\vspace*{.1in}
\centerline{Department of Mathematics and Center for Applied Mathematics}
\centerline{Cornell University, Ithaca, NY 14853}
}

\vspace*{.1in}
\begin{abstract}
\noindent
We consider both discrete and continuous
``uncertain horizon'' deterministic control processes,
for which the termination time is 
a random variable.
We examine the dynamic programming equations for
the value function of such processes,
explore their connections to infinite-horizon
and optimal-stopping problems,
and derive 
sufficient conditions
for the applicability of non-iterative (label-setting) methods.
In the continuous case, the resulting PDE has a free boundary,
on which all characteristic curves originate.
The causal properties of ``uncertain horizon'' problems can be
exploited to design efficient numerical algorithms:
we derive causal semi-Lagrangian and Eulerian discretizations
for the isotropic randomly-terminated problems, and use them
to build a modified version of the Fast Marching Method.
We illustrate our approach using numerical examples
from optimal idle-time processing and
expected response-time minimization.
\end{abstract}

\section{Introduction.}
\label{s:intro}

Deterministic and stochastic optimal control problems arise naturally
in most engineering disciplines and in fields as diverse as economics,
geometric optics, robotic navigation, and computational geometry.
Dynamic programming \cite{Bellman_PNAS} is the key technique for
solving such problems by formulating equations satisfied by the corresponding
{\em value function}. Fast algorithms for recovering the value function are
thus of interest to many practitioners, but the efficiency challenges
in building such methods can be quite different depending on the exact type
of optimal control problems.
One basic taxonomy is based on the time-horizon of optimization in the problem.
Is the process stopped at some explicitly specified terminal time
({\em finite-horizon problems}) or
continues forever ({\em infinite-horizon problems}) ?
If the terminal time is not specified,
but the process eventually stops, does this happen
upon entering some pre-specified set
({\em exit-time problems}),
or at any point when the controller chooses to do so
({\em optimal-stopping problems}) ?
For finite-horizon problems, the time-dependence of the value function
results in availability of simple and non-iterative (time-marching) numerical
methods.  In contrast, all other scenarios typically yield large systems
of coupled non-linear equations, and finding suitable non-iterative methods
for them can be challenging.
The above problem types are classical and the resulting equations are well
understood  (for readers' convenience we review them in
\S \ref{ss:graph_problem_types} and \S \ref{ss:contin_problem_types}
for discrete and continuous problems respectively).
In this paper we are interested in a less studied class
of {\em randomly-terminated} or {\em ``uncertain-horizon''} problems,
where the termination is assumed to be the result of a Poisson random process.
As we show here, such problems inherit some properties of both finite-horizon
and optimal-stopping cases.  Efficient numerical methods for them
are not as simple as time-marching, but we show that generalizations
of non-iterative algorithms previously derived for exit-time problems
are applicable.

Another natural classification approach is
to draw a distinction between deterministic and stochastic
control processes.
Not surprisingly, the former are usually simpler
and can be treated by more efficient numerical methods.
This observation holds for dynamic programming equations
both in discrete and continuous settings.

In the discrete cases, the deterministic Shortest Path (SP) problems
on graphs are often solved using fast-iterative ({\em label-correcting})
and non-iterative ({\em label-setting}) methods,
including the well-known Dijkstra's method \cite{Diks}
and its parallelizable Dial's variant \cite{Dial}.
We refer the readers to \cite{Ahuja, Bertsekas_NObook, Bertsekas_DPbook}
for a detailed discussion of these efficient methods on graphs.
The more general Stochastic Shortest Path (SSP) problems
typically require iterative methods; a good discussion of SSPs can be
found in \cite{Bertsekas_DPbook}.
Except for a few structurally simple examples,
the exact (a priori verifiable) conditions under
which non-iterative algorithms
are applicable to SSPs are still not known.
The second author has previously derived such sufficient conditions
for a narrower class of {\em Multimode} SSPs \cite{VladMSSP},
but the problems considered in the current paper lie outside of this class.

In the case of continuous state spaces, the deterministic problems
lead to first-order non-linear Hamilton-Jacobi-Bellman partial differential
equations; a comprehensive description can be found in \cite{BardiDolcetta}.
For exit-time problems,
the resulting PDE is typically static
since the starting time does not affect the optimality of any given control.
Efficient numerical methods for such static PDEs formed an active area of
research in the last fifteen years.
For example, Dijkstra-like non-iterative numerical methods for
isotropic problems were independently introduced by Tsitsiklis
\cite{Tsitsiklis_conference, Tsitsiklis} and
Sethian \cite{SethFastMarcLeveSet, SethSIAM}.  Later generalizations
lead to Ordered Upwind Methods \cite{SethVlad2, SethVlad3, AltonMitchell2}
applicable to anisotropic problems.  All these methods rely on
a careful use of Lagrangian information to efficiently solve
the Eulerian discretized equations,
yielding {\em space-marching} algorithms, in which the numerical solution
is marched on the grid inward from the boundary of the domain.

On the other hand, most commonly considered stochastic control problems
on continuous state space assume that the controlled dynamics is affected
by some time-continuous stochastic process
(usually, by a scaled Brownian motion).
This yields second-order semi-linear PDEs and, in the static case,
non-iterative numerical methods are not applicable.

The stochasticity considered in this paper is of a different kind --
the process dynamics and the running cost are assumed to be fully deterministic
up to the time of termination, but that terminal time is itself
a random variable.
Such uncertain-horizon problems have applications
in production/maintenance planning \cite{BoukasHaurieMichael},
economic growth \& global climate change modeling \cite{HaurieMoresino},
and multi-generational games \cite{Haurie}.
We first show that the optimal control of randomly-terminated problem
can be  always re-stated as a {\em time-discounted infinite-horizon}
problem on
the same state space.  (In the continuous case, this reformulation is
well-known; e.g., \cite{BoukasHaurieMichael}).
We then prove that efficient non-iterative methods are applicable
for a wide subclass of such problems.  Our exact technical assumptions
are specified in sections \ref{ss:graph_uncertain_horizon} and
\ref{ss:contin_value_func}, but this class generally
includes all problems where maintaining the ``status quo'' is always
possible and incurs a smaller immediate cost than any
attempt to change the system state.
That assumption allows re-stating this as a deterministic
{\em optimal stopping} problem.  We define a {\em motionless set} $\M$
of all states, starting from which the optimal behavior is to remain in place
(awaiting the termination).  Unfortunately, the set $\M$ is not
a priori known, which presents a challenge in using label-setting methods,
since $\B = \partial \M$ forms a {\em free boundary} of this problem,
from which the numerical solution should be marched into the rest
of the domain.

Our exposition proceeds on two parallel tracks:
the discrete and continuous settings are handled in sections
\ref{s:general_graph} and \ref{s:contin_space} respectively.
Some of the technical results are included in the Appendixes.
Throughout the paper the letters $v$ and $V$
are reserved for the value functions of
randomly-terminated problems,
while letters $u$ and $U$ are used for all other optimal
control examples.
We will also use the expressions ``deterministic randomly-terminated processes''
and ``uncertain-horizon processes'' interchangeably.
In \S \ref{ss:V_properties} and \S \ref{ss:contin_value_func}
we prove the properties of value functions and motionless sets,
including their asymptotic behavior depending on the probability
distribution for the terminal time.
We review the label-setting for the fully deterministic problems
on graphs in \S \ref{sss:dijkstra_SP} and then prove the applicability
of a Dijkstra-like method to uncertain-horizon problems in
\S \ref{sss:dijkstra_uncertain}.  The corresponding generalization
of the Fast Marching Method is developed in \S \ref{ss:FMM_modified}.
Numerical experiments illustrating the properties of the latter
are included in \S \ref{s:examples}; several possible 
generalizations are discussed in \S \ref{s:conclusions}.


\section{Uncertain-horizon problems on a finite state space.}
\label{s:general_graph}

Optimal control problems on a finite state space
are usually presented as problems of finding
optimal paths in a finite directed graph.
The notion of ``optimality'' here depends not only on
edge costs, but also on termination conditions for the process.

We begin by briefly reviewing several
common types of such problems to put
uncertain-horizon optimization in context.

\subsection{Common optimal control problems on graphs.}
\label{ss:graph_problem_types}

We will assume that all paths are considered on a directed graph
with $M$ nodes $X = \{x_1, \ldots, x_M\}$.
Let $N(x_i) \subset X$ denote the set of nodes
to which a direct transition from $x_i$ is possible,
and assume that $\kappa << M$ is an upper bound on the outdegrees of all nodes
(i.e., $|N(x_i)| \leq \kappa$ for $\forall x_i \in X$).

Suppose $y_k \in X$ is the position after $k$ steps.
A sequence $\y = (y_0, \ldots, y_k, \ldots)$ is a {\em path} on this graph
if $y_{k+1} \in N(y_k)$ for each $k \geq 0$.
Each transition
from $x_i$ to $x_j$ incurs a cost $K_{ij} = K(x_i,x_j)$
(assumed to be $+\infty$ if $x_j \not \in N(x_i)$).
If the termination occurs at $x_j$, this results in an additional terminal-penalty $q_j = q(x_j).$

\begin{itemize}
\item
{\em Finite-horizon problems:}\\
The process runs for exactly $t$ steps (where $t$ is a constant
specified as a part of problem description).
The total cost of a path $(y_0, \ldots, y_t)$ is
\begin{equation}
\label{eq:finite_path_cost}
Cost(y_0, \ldots, y_t) = \sum\limits_{k=0}^{t-1} K(y_k, y_{k+1}) \; + \; q(y_t).
\end{equation}
The central idea of {\em dynamic programming} is to define
the {\em value function} $U(x_i,k) = U_i^k$ as the minimum cost to spend starting
from $x_i$ with $(t-k)$ steps remaining.
Bellman's optimality principle \cite{Bellman_PNAS}
allows to derive equations for $U_i^k$ using the values
in $\{ U_j^{k+1} \mid x_j \in N(x_i) \}$ only.
In particular,
\begin{equation}
\label{eq:DP_finite_horizon}
\begin{cases}
U_i^k = \min\limits_{x_j \in N(x_i)}
\left\{ K_{ij} + U_j^{k+1} \right\},
&\text{ for } k=0,\ldots,(t-1) \text{ and } i=1,\ldots,M;\\
U_i^t = q_i \text{ for } i=1,\ldots,M.
\end{cases}
\end{equation}
It is easy to generalize this to use time-dependent $K(x_i,x_j,k)$ and
the system (\ref{eq:DP_finite_horizon}) can be efficiently solved
in a single backward-sweep (from $k=t$ to $k=0$), regardless of
the sign of transition cost $K$.

\item
{\em Exit-time a.k.a. shortest path (SP) problems:}\\
Here the termination occurs immediately upon reaching
the exit-set $Q \subset X$; i.e., $t_{\y} = min \{k | y_k \in Q\}$.
As a result, $U_i = U(x_i)$ (the min-cost-to-exit-starting-from-$x_i$)
no longer depends on the number of steps already used to reach $x_i$.
The value function is defined only in the absence of ``negative cost cycles''
and, by optimality principle, must satisfy
\begin{equation}
\label{eq:DP_classic_SP}
\begin{cases}
U_i = \min\limits_{x_j \in N(x_i)} \left\{K_{ij} + U_j\right\},
&\text{ for } \forall x_i \in X \backslash Q;\\
U_i = q_i,
&\text{ for } \forall x_i \in Q.
\end{cases}
\end{equation}
This is a system of $M$ coupled non-linear equations,
and, unless the graph is acyclic, this system might be
expensive to solve iteratively.  If all $K_{ij}$'s are non-negative,
label-setting methods (further discussed in section \ref{sss:dijkstra_SP})
provide an efficient way of solving it.

\item
{\em Infinite-horizon time-discounted problems:}\\
Here the process never terminates -- the paths are infinite,
but the cost of each subsequent step is discounted, i.e.,
\begin{equation}
\label{eq:infinite_path_cost}
Cost(y_0, \ldots) = \sum\limits_{k=0}^{+\infty} \alpha^k \, K(y_k, y_{k+1}),
\end{equation}
where $\alpha \in (0,1)$ is the discounting factor.
The value function $U_i = U(x_i)$ is well-defined for all nodes
regardless of sign of $K_{ij}$'s and satisfies
\begin{equation}
\label{eq:DP_inf_horizon}
U_i = \min\limits_{x_j \in N(x_i)} \left\{ K_{ij} + \alpha U_j \right\}
\qquad
\text{ for } \forall x_i \in X.
\end{equation}
Unlike the SP, in this case optimal paths might include cycles.

\item
{\em Optimal stopping problems:}\\
Same as the infinite-horizon problem, but with an option
of deciding to terminate the process at any node $x$
by paying an exit time-penalty $q(x)$.
The value function then satisfies
\begin{equation}
\label{eq:DP_opti_stop}
U_i = \min\left( q_i, \;
\min\limits_{x_j \in N(x_i)} \left\{ K_{ij} + \alpha U_j \right\}
\right)
\qquad
\text{ for } \forall x_i \in X.
\end{equation}
We note that, unlike the previous case, the value function is
well-defined even with $\alpha = 1$ (i.e., without time-discounting).
\end{itemize}


\subsection{The value function for uncertain-horizon processes on graphs.}
\label{ss:graph_uncertain_horizon}

In this paper we concentrate on {\em uncertain-horizon problems},
where the process does not terminate deterministically after
a pre-specified number of steps nor upon reaching
some special exit set $Q$, but instead might terminate
with probability $p \in (0,1)$ after each transition\footnote{
More general problems with ($p$ dependent on the current state or
even on the last transition) can be handled very similarly; in this section
we use the same constant $p$ at all nodes for the sake of notational simplicity.}.
Two simple examples of such randomly-terminated problems are
provided in Figure \ref{fig:discrete_examples}.


We will use $Y$ to denote the set of all infinite paths
on $X$ and $Y(x) = \{ \y \in Y \, \mid \, y_0 = x \}$ for
the set of all paths starting from $x$.
If the termination occurs after $t$ steps,
the total cost of the path is the same as in (\ref{eq:finite_path_cost}).
The a priori probability of termination after exactly
$t \geq 1$ steps is clearly $\hat{P}_t = (1-p)^{t-1} p$.
The cost of an infinite path $\y = (y_0, \ldots)$ is now a random
variable whose expected value is
\begin{equation}
\label{eq:path_cost}
J(\y) = E[Cost(\y)] =
\sum_{t=1}^{\infty} \hat{P}_t \, Cost(y_0, \dots, y_t).
\end{equation}
Since the graph is finite, functions $K$ and $q$ are bounded,
the above series is absolutely convergent, and
$|J(\y)|$ is uniformly bounded for all $\y \in Y$.\\

Starting the process from any node $x \in X$, the goal is to minimize the expected total cost
up to the termination.  The value function $V(x)$ can be defined as
\begin{equation}
\label{eq:value_through_paths}
V(x) = \min\limits_{\y \in Y(x)} J(\y),
\end{equation}
where the existence of a minimizer follows from the compactness of $Y(x)$
and continuity of $J$; see Lemma \ref{lm:min_attained} in the Appendix A.

We will say that $\y = (y_0, y_1, \ldots) \in Y$ is a {\em simple path} if
$y_k = y_m$ implies $y_{k+1} = y_{m+1}$.
A {\em simple loop} is a periodic simple path.
Since the set $X$ is finite,
any simple path leads to a simple loop within the first $M$ steps.
We will use $Y^s$ to denote the set of all simple paths and $Y^s(x) \subset Y^s$
for the set of all such paths starting from $x$.
Theorem \ref{thm:stationary} in the Appendix A proves the existence of a
{\em minimizing simple path} for every $x \in X$.

We note that on any path at least one transition happens before the termination
with probability one.  As a result we can rewrite
\begin{equation}
\label{eq:path_cost_recursive}
J(\y)
= K(y_0, y_1) + p q(y_1) + \sum_{t=2}^{\infty} \hat{P}_t \, Cost(y_1, \dots, y_t)
= K(y_0, y_1) + p q(y_1) + (1-p) J(y_1, \dots).
\end{equation}
This yields the Optimality Principle:
\begin{equation}
\label{eq:DP_uncertain}
V_i = V(x_i) = \min\limits_{x_j \in N(x_i)}
\left\{ K_{ij} + p q_j + (1-p) V_j \right\},
\hspace*{2cm} i=1, \ldots, M.
\end{equation}

\begin{remark}
\label{rem:relation_to_infinite_horizon}
This problem can also be restated as an infinite-horizon deterministic
control by setting the costs of transition
$\tilde{K}_{ij} = K_{ij} + p q_j$ and the discounting factor $\alpha = (1-p)$.
Indeed, the expected cost of any uncertain-horizon path can be re-written as
\begin{eqnarray}
\nonumber
E[Cost(\y)] &=& \sum_{k=0}^{\infty} K(y_k, y_{k+1}) \left( \sum_{t=k+1}^{\infty} \hat{P}_t \right)
\; + \; \sum_{t=1}^{\infty} q(y_t) \hat{P}_t
\; = \; \sum_{k=0}^{\infty} K(y_k, y_{k+1}) (1-p)^k \; + \; \sum_{t=1}^{\infty} q(y_t) \hat{P}_t\\
&=& \sum_{k=0}^{\infty} \left( K(y_k, y_{k+1}) + p q(y_{k+1}) \right) (1-p)^k
\; = \; \sum_{k=0}^{\infty} \alpha^k \tilde{K}(y_k, y_{k+1}).
\label{eq:infinite_horizon}
\end{eqnarray}

Conversely, consider an infinite-horizon deterministic
control problem on a graph specified by a matrix of $\tilde{K}_{ij}$'s
and by a discounting factor $\alpha \in (0,1)$.
Assume that self-transitions are allowed at every node, i.e.,
$$
\Assume{1} \hspace*{2cm}
x_i \in N(x_i) \qquad \text{ for } \forall x_i \in X.
$$
(Whenever we need to refer to all {\em other} transitions possible
from $x_i$, we will use the set $\hat{N}(x_i) = N(x_i) \backslash \{x_i\}$.)
Subject to assumption \Assume{1}, each infinite-horizon problem can also be restated as
a randomly-terminated problem by setting:
$$
p = (1 - \alpha); \quad q_{i} = \tilde{K}_{ii}/p;
\quad K_{ij} = \tilde{K}_{ij} - p q_j; \quad \quad K_{ii} = 0.
$$
\end{remark}

\begin{remark}
\label{rem:free_self_transition}
Similarly, if the assumption \Assume{1} holds in an uncertain-horizon problem,
we can always assume that self-transitions incur zero cost by setting
$$
q^{new}_i = q_i + K_{ii} / p;
\qquad
K_{ij}^{new} = K_{ij} + p \, (q_j - q^{new}_j) = K_{ij} - K_{jj};
\qquad
K_{ii}^{new} = 0.
$$
Thus, for the rest of this paper we will assume
without loss of generality that
$$
\Assume{2} \hspace*{2cm}
K_{ii}=0 \text{ for all } i=1,\ldots,M.
$$
\end{remark}

We will also make an additional assumption,
the computational consequences of which will become clear in section
\ref{sss:dijkstra_uncertain}:
$$
\Assume{3} \hspace*{2cm}
K_{ij} \geq \Delta \geq 0 \qquad \text{ for all } i=1,\ldots,M
\text{ and for all } j \neq i.
$$

More generally, we can similarly treat any randomly-terminated problem,
where a transition from another node costs at least as much
as staying in place.  (Indeed, if \Assume{1} holds and
$K_{ij} \geq K_{jj}$ for all $i$ and $j$,
then the procedure described in Remark \ref{rem:free_self_transition}
will ensure that $K^{new}$ will satisfy \Assume{2} and \Assume{3}.)

In view of \eqref{eq:path_cost}, if a constant is added to all $q_i$'s,
then the same constant is added to all $V_i$'s.  As a result,
we can also assume without loss of generality that
$q_i > 0$ for $\forall x_i \in X.$

\begin{figure}[hhhh]
\center{
$
\hspace*{-10mm}
\begin{array}{cc}
\begin{tikzpicture}[->,>=stealth',shorten >=1pt,auto,node distance=3cm,
                    semithick]
  \tikzstyle{every state}=[draw, shape=circle]

  \node[state]         (A)								{$x_1$};
  \node					   [below of=A, yshift=22mm]	{$q_1=0$};
  \node[state]         (B) [right of=A]					{$x_2$};
  \node					   [below of=B, yshift=22mm]	{$q_2=0$};

  \path (A) edge	[bend left=45]		node		{$K_{12}=1$}  (B)
            edge	[loop left]			node [swap] {$K_{11}=10$} (A)
        (B) edge	[bend left=45]		node		{$K_{21}=1$}  (A)
            edge	[loop right]		node [swap]	{$K_{22}=10$} (B);
\end{tikzpicture}&
\hspace*{10mm}
\begin{tikzpicture}[->,>=stealth',shorten >=1pt,auto,node distance=2cm,
                    semithick]
  \tikzstyle{every state}=[draw, shape=circle]

  \node[state]         (A)								{$x_1$};
  \node					   [below of=A, yshift=12mm]	{$q_1=1$};
  \node[state]         (B) [right of=A]					{$x_2$};
  \node					   [below of=B, yshift=12mm]	{$q_2=10$};
  \node[state]         (C) [right of=B]					{$x_3$};
  \node					   [below of=C, yshift=12mm]	{$q_3=0$};
  \node					   [above of=B, yshift=-7mm]	{$K_{ij} = 0 \, $ for all $ \, i,j$};

  \path (A) edge (B)
			edge [loop above] (A)
        (B) edge (C)
			edge [loop above] (B)
		(C)	edge [loop above] (B);
\end{tikzpicture}\\
(A) & \hspace*{10mm} (B)
\end{array}
$
}
\caption{
{\footnotesize
{\bf Left: }
A simple example violating assumption \Assume{2}.
The optimal path clearly alternates between $x_1$ and $x_2$.
By the symmetry, $V_1 = V_2 = 1/p$.  We note that the value iterations
process described in \S \ref{ss:label_setting} would generally require
infinitely many iterations here (except for a lucky initial guess $W^0 = V$).
{\bf Right: } A simple example satisfying assumptions \Assume{1}-\Assume{3}.
We note that $V_2 = V_3 = q_3 = 0$ and, starting from $x_2$
the path $(x_2, x_3, x_3, \ldots)$ is always optimal regardless of $p$.
However, starting from $x_1$, the path $(x_1, x_2, x_3, x_3, \ldots)$
is optimal only for $p \leq 0.1$, and $V_1 = \min (1, \, 10p)$.
}
}
\label{fig:discrete_examples}
\end{figure}
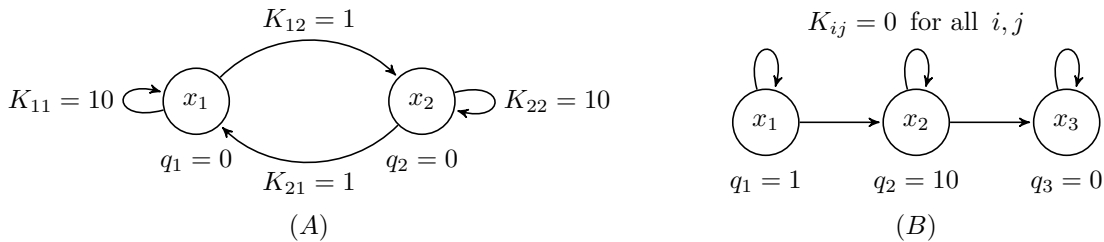

\subsection{Properties of the value function.}
\label{ss:V_properties}
Figure \ref{fig:discrete_examples}A shows a simple example, where every optimal path is a loop.
As we'll see later, such loops make it impossible to solve
the system \eqref{eq:DP_uncertain} efficiently.
Fortunately, conditions \Assume{1} - \Assume{3} preclude this scenario
and also yield useful bounds on $V$.

We will refer to a path $\y = (y_0, y_1, \ldots)$ as {\em eventually motionless}
(or {\em e.m.}) if there exists
some $m \geq 0$ such that $y_k = y_m$ for all $k \geq m$.
We note that for any such e.m. path the formula \eqref{eq:path_cost}
can be rewritten as
\begin{equation}
\label{eq:path_cost_terminal}
J(\y) \; =  \;
\sum_{t=1}^{m-1} (1-p)^{t-1} p \, Cost(y_0, \dots, y_t)
\; + \; (1-p)^{m-1} Cost(y_0, \dots, y_m).
\end{equation}
We will refer to a node $x \in X$ as {\em motionless}
if $\y = (x, x, x, \ldots)$ is a valid path and $V(x) = J(\y) = q(x)$.
We will also use $\M \subset X$ to denote a set of all such motionless nodes.
To illustrate, in the example of Figure \ref{fig:discrete_examples}A
$\M = \emptyset$ regardless of $p$.
In the example of Figure \ref{fig:discrete_examples}B,
$\x_3 \in \M, \, \x_2 \not \in \M$, but $\x_1 \in \M$ only if $p \geq 0.1$.



\begin{thm}
Suppose an uncertain-horizon problem is such that
\Assume{1}, \Assume{2}, and \Assume{3} hold.  Then
\begin{enumerate}
\item
$V_i \leq q_i$ for all $x_i \in X$.

\item
If $\y \in Y^s(x_i)$ is optimal and $y_k = x_j$
then $V_i \geq k \Delta + V_j$.

\item
If a loop $(y_0, \ldots, y_{m})$ is an optimal path, then
$V(y_k) = q(y_k) = q(y_0)$ for all $k=0,\ldots,m.$

\item
For every $x \in X$ there exists an e.m. optimal path $\y \in Y^s(x)$.

\item
If an e.m. optimal path from $x$
leads through $\xtilde$, then there exists
an e.m. optimal path from $\xtilde$ avoiding $x$.
\end{enumerate}
\label{thm:causality}
\end{thm}
\begin{proof}

\begin{enumerate}
\item
Assumption \Assume{2} and (\ref{eq:DP_uncertain}) yield
$$
V_i \leq p q_i + (1-p) V_i \quad \Longrightarrow \quad
V_i \, \leq \, q_i = J(x_i,x_i, x_i, \ldots).
$$

\item
First, suppose $y_1 = x_j$; i.e.,
$x_j \in N(x_i)$ is a minimizer in (\ref{eq:DP_uncertain}).
Then, using the above,
$
V_i = K_{ij} + p q_j + (1-p) V_j \geq K_{ij} + V_j \geq \Delta + V_j$.
Applying this argument recursively yields $V_i \geq k \Delta + V_j$.

\item
An optimal path is optimal for every node on it.
Suppose $q(y_n) \leq q(y_k)$ for $k=0,\ldots,m$.  Then $\xtilde = y_n \in \M$
: given \Assume{3}, the path $\tilde{\y} = (\xtilde, \xtilde, \xtilde, \ldots)$
is at least as good for $\xtilde$ as the original loop $\y$.
Also $V(\xtilde) = J(\tilde{\y}) = q(\xtilde)$.
The equality $J(\y) = J(\tilde{\y})$
is only possible if $K(y_k, y_{k+1}) = 0$ and $q(y_k) = q(\xtilde)$
for $k=0,\ldots,m$.  As a result, every $y_k \in \M$
and $V(y_0) = \ldots = V(y_m)$.
(This also shows that a non-trivial loop cannot be optimal if $\Delta > 0$.)

\item
Theorem \ref{thm:stationary} shows the existence of an optimal simple path $\y$
(eventually leading to a simple loop).  The above also shows that
the entry point of that simple loop $y_k$ is a motionless node.
Thus, the loop can be replaced by $(y_k, y_k, \ldots)$ without changing
the cost of $\y$.

\item
For $\Delta>0$ this is obvious (if an optimal path from $x$ passes
from $\xtilde$, then $V(x) > V(\xtilde))$. If $\Delta=0$,
then combining segments of both optimal paths we can form an optimal loop
passing through both $x$ and $\xtilde$.  As shown above, this implies
that $\xtilde$ is a motionless node.
\end{enumerate}
\end{proof}


In the rest of this subsection we will use superscripts
to indicate the dependence on $p$ of the expected path-cost $J^p(\y)$,
the value function $V^p(x)$ and
the motionless set $\M^p = \{x \mid V^p(x) = q(x) \}$.
Wherever this superscript is omitted,
the properties hold for any fixed $p \in (0,1)$.
To address two extreme cases, we introduce $V^0$ and $V^1$
as respective solutions of two systems:
\begin{eqnarray}
\label{eq:DP_V0}
V^0_i = V^0(x_i) &=& \min \left\{ \min\limits_{x_j \in \hat{N}}
\left\{ K_{ij} +  V^0_j \right\}, \, q_i \right\},
\hspace*{2cm} i=1, \ldots, M;\\
\label{eq:DP_V1}
V^1_i = V^1(x_i) &=& \min\limits_{x_j \in N(x_i)}
\left\{ K_{ij} +  q_j \right\},
\hspace*{2cm} i=1, \ldots, M.
\end{eqnarray}

We see that \eqref{eq:DP_V0} is a variant of \eqref{eq:DP_opti_stop}
with $\alpha = 1$.  In other words, $V^0$ is the value function for
a deterministic  optimal stopping problem with no time-discounting.
As a result, the label-setting methods of section \ref{sss:dijkstra_SP}
are applicable, and for every starting node $x$ there exists
some optimal finite path $(y_0, \ldots, y_m)$;
i.e., $y_0 = x$ and $V^0(x) = Cost(y_0, \ldots, y_m)$.
In the framework of randomly-terminated problems,
$V^0$ can be intuitively interpreted as the limiting case,
in which the termination is so unlikely that we are guaranteed to reach
any node of our choice before it occurs.  (This interpretation is
justified in Theorem \ref{thm:asymptotics_V} below).

$V^1$ corresponds to the opposite case, where the termination
always happens after the very first transition.
We note that since $x_i \in N(x_i)$,
we have $V^1_i \leq q_i$ for all $x_i \in X.$
By the above definition,
if $x_i \in \M^1$,  we have $q_i \leq K_{ij} +  q_j$
for all $x_j \in \hat{N}(x_i)$.  An important subset
is obtained when the self-transition is {\em the only}
optimal strategy:
$$
\M^1_0 = \left\{ x_i \in X \, \mid \,
q_i < K_{ij} +  q_j, \, \forall x_j \in \hat{N}(x_i)
\right\}.
$$
Below we show that $\M^p \subset \M^1, \forall p \in (0,1)$.
However the example in Figure \ref{fig:subtle_motionless}
shows that some nodes in $\M^1\backslash \M^1_0$
might become motionless only in the limit (as $p \to 1$).
On the other hand, every node in $\M^1_0$ is already
motionless for some $p < 1$.

\begin{figure}[hhhh]
\center{
\begin{tikzpicture}[->,>=stealth',shorten >=1pt,auto,node distance=3cm,
                    semithick]
  \tikzstyle{every state}=[draw, shape=circle]

  \node[state]         (A)								{$x_1$};
  \node					   [below of=A, yshift=22mm]	{$q_1=10$};
  \node[state]         (B) [right of=A]					{$x_2$};
  \node					   [below of=B, yshift=22mm]	{$q_2=9$};
  \node[state]         (C) [right of=B]					{$x_3$};
  \node					   [below of=C, yshift=22mm]	{$q_3=0$};

  \path (A) edge			  node {$K_{12}=1$} (B)
			edge [loop above] node {$K_{11}=0$} (A)
        (B) edge			  node {$K_{23}=C$} (C)
			edge [loop above] node {$K_{22}=0$}  (B)
		(C)	edge [loop above] node {$K_{33}=0$}  (B);
\end{tikzpicture}
}
\caption{
{\footnotesize
A simple example with subtle motionless set properties.
First, note that $x_1 \in \M^1 \backslash \M^1_0$ for all $C>0$.
Second, assuming that $C=1$, we have $V^p_3 = 0, \, V^p_2 = 1, \
V^p_1 = 2 + 8p,$ and $x_1 \not \in \M^p$ for any $p < 1.$
On the other hand, if $C >9$ then $x_1 \in \M^p$
for all $p \in[0,1]$.
}
}
\label{fig:subtle_motionless}
\end{figure}
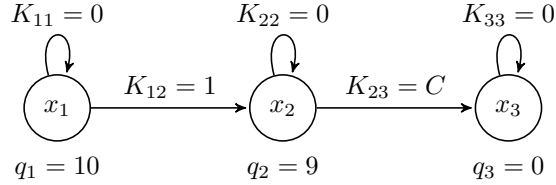

To simplify the notation, we will refer to a path $\y \in Y(x)$ as
{\em $p$-optimal} if $J^p(\y) = V^p(x)$.
\begin{thm}
Suppose an uncertain-horizon problem is such that
\Assume{1}, \Assume{2}, and \Assume{3} hold.  Then
\begin{enumerate}

\item
Suppose $\y = (y_0, \ldots, y_m, y_m, \ldots)$ is a $p$-optimal e.m. path
leading from $x=y_0$ to $y_m$.  Define the cost of each finite
subpath as in \eqref{eq:finite_path_cost}.
Then $Cost(y_0, \ldots y_m) \leq Cost(y_0, \ldots y_n)$
and $q(y_m) \leq q(y_n)$ for all $n<m$ (and the second inequality becomes strict
if $\Delta >0$).

\item
$V^p(x) \in [V^0(x), \, V^1(x)]$ for all $p \in (0,1)$ and all $x \in X$.

\item
$0 < p_1 \leq p_2 < 1 \qquad \imply \qquad V^{p_1}(x) \leq V^{p_2}(x),
\quad \forall x \in X.$

\item
$V^p(x_i) \to V^1(x_i)$ for all $x_i \in X$ as $p \to 1$.

\item
$V^p(x_i) \to V^0(x_i)$ for all $x_i \in X$ as $p \to 0$.

\item
If $p_1 \leq p_2$, then $\M^{p_1} \subset \M^{p_2}.$
In particular, $\M^0 \; \subset\;  \M^{p} \; \subset \; \M^1$
for $\forall p \in (0,1).$

\item
For every $x_i \in \M^1_0$, there exists a sufficiently large
$p < 1$, such that $x_i \in \M^{p}$.

\item
For every $x_i \not \in \M^0$, there exists a sufficiently small
$p >0$, such that $x_i \not \in \M^p$.

\item
Define the sets of local and global minima of $q$:
$$
Q_l = \{ x_i \mid q(x_i) \leq q(x_j),  \, \forall x_j \in N(x_i) \};
\qquad
Q_g = \{ x_i \mid q(x_i) \leq q(x_j),  \, \forall x_j \in X \}.
$$
Then $Q_g \subset \M^0, \quad Q_l \subset \M^1,
\quad \text{and} \quad  \forall K_{ij} = 0 \, \imply \, Q_l = \M^1.$
\end{enumerate}
\label{thm:asymptotics_V}
\end{thm}
\begin{proof}
\begin{enumerate}

\item
Let $l = \argmin_{n\leq m} Cost(y_0, \ldots, y_n)$ and
define $\tilde{\y} = (y_0, \ldots, y_l, y_l, \ldots)$.
If $Cost(y_0, \ldots, y_l) < Cost(y_0, \ldots, y_m)$, then
the formula \eqref{eq:path_cost} implies that $J^p(\tilde{\y}) < J^p(\y)$,
which contradicts the $p$-optimality of $\y$.
(We emphasize that the above argument {\em does not} imply that
$Cost(y_0, \ldots, y_n)$ is monotone non-increasing in $n$ along
every $p$-optimal path; see Figure \ref{fig:discrete_examples}B
for a counter-example.)
We note that
$$
Cost(y_0, \ldots, y_n) \; = \;
\sum\limits_{k=0}^{n-1} K(y_k, y_{k+1}) \; + \; q(y_n)
\; \geq \; Cost(y_0, \ldots, y_m) \geq
\sum\limits_{k=0}^{n-1} K(y_k, y_{k+1}) \; + \; q(y_m),
$$
where the second inequality uses \Assume{3} and becomes strict if
$\Delta > 0$.  Thus, $q(y_m) \leq q(y_n)$.

\item
Let $\y = (y_0, \ldots, y_m, y_m, \ldots)$ be a $p$-optimal e.m.
path leading from $x$ to $y_m$.  Note that \eqref{eq:path_cost}
defines $J^p(\y)$ as a weighted average of $Cost(y_0, \ldots, y_n)$ terms.
Then, from the previous part,
$J^p(\y) \geq Cost(y_0, \ldots, y_m) \geq V^0(x)$.
On the other hand, $V^1$ can also be defined as
$V^1(x) = \min J^p(\y)$, minimizing over all ``infinite''
paths $\y \in Y^s(x)$ of the form $\y=(x, y_1, y_1, \ldots)$;
hence, $V^p(x) \leq V^1(x)$.

\item
Fixing $p_2$, we define $m(x)$ to be the minimum number of transitions
before reaching the motionless node along the $p_2$-optimal paths from $x$.
(E.g., $x \in \M^{p_2} \; \Longleftrightarrow \; m(x) = 0$.)
We prove this statement by induction on $m(x)$.
First, for $m(x) = 0$, we have $V^{p_2}(x) = q(x) \geq V^{p_1}(x)$.
Now suppose the statement is true for all $x_j$ such that $m(x_j) \leq k$
and consider $x_i$ such that $m(x_i) = k+1.$
Then there exists $x_j^* \in N(x_i)$, a minimizer in
formula \eqref{eq:DP_uncertain} such that $m(x_j^*) = k$.  As a result,
\begin{eqnarray*}
V^{p_2}(x_i) &=& K_{i j^*} + p_2 q(x_j^*) + (1-p_2) V^{p_2}(x_j^*)
\; \geq \;
K_{i j^*} + p_2 q(x_j^*) + (1-p_2) V^{p_1}(x_j^*)\\
&=&
K_{i j^*} + V^{p_1}(x_j^*) + p_2 \left( q(x_j^*) - V^{p_1}(x_j^*) \right)
\; \geq \;
K_{i j^*} + V^{p_1}(x_j^*) + p_1 \left( q(x_j^*) - V^{p_1}(x_j^*) \right)\\
&\geq&
V^{p_1}(x_i).
\end{eqnarray*}

\item
For any $\y \in Y(x)$, formulas
\eqref{eq:finite_path_cost} and \eqref{eq:path_cost} can be combined to show that,
$$
J^p(\y) = p \left[
K(y_0,y_1) + q(y_1) \right]
\, + \, \sum_{t=2}^{\infty} (1-p)^{t-1} p \, Cost(y_0, \dots, y_t)
\; \to  \; \left[K(y_0,y_1) + q(y_1) \right] \text{ as } p \to 1.
$$

\item
For any e.m. $\y=(y_0, \ldots, y_m, y_m, \ldots) \in Y(x_i)$,
the formula \eqref{eq:path_cost_terminal} implies that
$J^p(\y) \, \to  \, Cost(y_0, \dots, y_m)$  as $p \to 0.$
If the finite path $(y_0, \dots, y_m)$ is optimal for $V^0(x_i)$,
then  $Cost(y_0, \dots, y_m) = V^0(x_i) \leq V^p(x_i) \leq J^p(\y)$,
implying $V^p(x_i) \to V^0(x_i).$

\item
If $x \in \M^{p_1}$, then $q(x) = V^{p_1}(x) \leq V^{p_2}(x) \leq q(x)$.
Thus, $x \in \M^{p_2}.$

\item
Since $x_i \in \M^1_0$, we can define
$$
p \; = \; q_i \; / \,
\min\limits_{x_j \in \hat{N}(x_i)} \{ K_{ij} + q_j \}
\; < \, 1.
$$
As a result, for any simple path $\y \in Y^s(x_i)$
involving at least one non-trivial transition (i.e., with $y_1 \neq y_0=x_i$),
we have
$J^p(\y) \geq p \left[
K(y_0,y_1) + q(y_1)
\right] \geq q_i. \; $  Thus, $x_i \in \M^p.$

\item
Part 5 of this Theorem implies $\lim\limits_{p \to 0} V^p(x_i) = V^0(x_i) < q_i$.
Thus, for sufficiently small $p$, $V^p(x_i) < q_i$, and $x_i \not \in \M^p.$

\item
If $x_i \in Q_g$, then for any path $(y_0, \ldots, y_m)$
starting from $y_0=x_i$, we have
$Cost(y_0, \ldots, y_m) \geq q(y_m) \geq q_i \; \imply \; x_i \in \M^0. \,$
If $x_i \in Q_l$, then $q_i \leq q_j \leq K_{ij} + q_j$ for all
$x_j \in N(x_i) \; \imply \; x_i \in \M^1.\,$

If $\forall K_{ij} = 0$,
the definitions of $\M^1$ and $Q_l$ coincide.  In this case
every $p$-optimal trajectory becomes motionless at some local minimum of $q$.
\end{enumerate}
\end{proof}

\begin{remark}
For the general infinite-horizon and/or optimal stopping
problems reviewed in \S \ref{ss:graph_problem_types}, the cost of any trajectory
is obviously a non-decreasing function of the discounting factor
$\alpha$.  We note that part 3 of the above theorem does not contradict
this monotonicity, since the procedure described
in Remark \ref{rem:relation_to_infinite_horizon}
also makes $\tilde{K}_{ij}$'s dependent on $p$.
\end{remark}

\subsection{Solving the Dynamic Programming equations.}
\label{ss:label_setting}
The Optimality Principle (\ref{eq:DP_uncertain}) yields a system of $M$
coupled non-linear equations, which can be challenging to solve efficiently.

An operator $T$ can be defined on $\R^M$ component-wise by applying
the right hand side of equation (\ref{eq:DP_uncertain}); i.e., for
any $W \in \R^M$,
\begin{equation}
(TW)_i \; = \;
\min\limits_{x_j \in N(x_i)}
\left\{ K_{ij} + p q_j + (1-p) W_j \right\}.
\label{eq:generic_value_it_component}
\end{equation}
Clearly,
$V =
\left[
\begin{array}{c}
V_1\\
\vdots\\
V_M
\end{array}
\right]
$
is a fixed point of $T$ and one hopes to recover $V$
by {\em value iteration}:
\begin{equation}
W^{n+1} \, := \, T \, W^n \qquad
\text{ starting from an initial guess $W^0 \in \R^M$.}
\label{eq:generic_value_it}
\end{equation}

The value iterations will converge (i.e., $W^n \rightarrow V$) for any initial guess $W^0$
\cite{BertsTsi}, but this approach is not very efficient.
It is generally possible that the value iterations will not converge in a finite number
of steps (e.g., see the example in Figure \ref{fig:discrete_examples}A).
If the problem is such that the convergence is attained in finitely many steps
for every $W^0$, it is easy to show that $W^M = V$
(i.e., at most $M$ value iterations will be needed).
Since the computational cost of each value iteration is $O(M)$,
this results in the overall cost of $O(M^2)$.
A standard Gauss-Seidel relaxation can be used to speed up the convergence, but then
the efficiency becomes strongly dependent on the chosen node-ordering and the worst case
computational cost remains $O(M^2)$.

Label-setting methods provide an attractive alternative to value iterations for problems
with intrinsic {\em causality}.
These methods reorder the iterations over
the nodes to guarantee that each $V_i$ is recomputed at most
$\kappa << M$ times.
This can also be interpreted as finding the optimal Gauss-Seidel ordering
{\it dynamically} -- based on the value function already correctly computed on a part of the graph.

\subsubsection{Label-setting for SP problems.}
\label{sss:dijkstra_SP}

For reader's convenience, we first provide a brief overview
of standard label-setting methods for deterministic shortest path problems (\ref{eq:DP_classic_SP}).
If we assume $K_{ij} \geq \Delta \geq 0,$
then $U_i = K_{ij} + U_j$ implies $U_i \geq \Delta + U_j \geq U_j$.
This induces a causal ordering: each $U_i$ depends only on the smaller values
in the adjacent nodes.  I.e., if $x_j \in N(x_i)$ and $U_j > U_i$
then replacing $U_j$ with $+\infty$ will not change the value of $U_i$ as
computed by formula (\ref{eq:DP_classic_SP}).
This is the basis of Dijkstra's classic method \cite{Diks}.

The method subdivides $X$ into three classes: $Accepted$
(or ``permanently labeled'') nodes,
$Considered$ (or ``tentatively labeled'') nodes that
have $Accepted$ nodes among their neighbors, and
$Far$ (or ``unlabeled'') nodes. The values for $Considered$ $x_i$'s are
successively re-evaluated using only the previously $Accepted$ adjacent values:
\begin{equation}
\label{eq:temp_labels_Determ}
U(x_i) \; := \; \min\limits_{x_j \in \tilde{N}(x_i)} \left\{K_{ij} + U_j\right\},
\end{equation}
where $\tilde{N}(x_i) = N(x_i) \cap Accepted$.
The algorithm is initialized by designating all exit-nodes as $Considered$
(with tentative labels $U(x_i) = q_i$ for all $x_i \in Q$) and all other nodes
as $Far$ (with tentative labels $U(x_i) =+\infty$ for all $x_i \in X \backslash Q$).
At each stage the algorithm chooses the smallest of tentative labels $U(\xbar)$,
designates $\xbar$ as $Accepted$ (making this label permanent and removing $\xbar$
from the list of $Considered$), and
re-evaluates $U_i$ for each not-yet-$Accepted$ $x_i$ such that $\xbar \in N(x_i).$
Since $\xbar$ is the only new element in $\tilde{N}(x_i)$, that re-evaluation
can be more efficiently performed by setting
\begin{equation}
U(x_i) \; := \; \min \left\{ U(x_i), \left( K(x_i, \xbar) + U(\xbar) \right) \right\}.
\label{eq:temp_labels_local_Determ}
\end{equation}
In addition, all previously $Far$ $x_i$'s updated at this step are designated $Considered$.\\
The algorithm terminates once the list of $Considered$ nodes is empty, at which
point the vector $U \in \R^M$ satisfies the system of equations
(\ref{eq:DP_classic_SP}).
The necessity to sort all tentative labels suggests
the use of heap-sort data structures for $Considered$ nodes \cite{Ahuja},
resulting in the overall computational complexity of $O(M \log M)$.

If $\Delta > 0$, then Dial's label-setting method is also
applicable \cite{Dial}. The idea is to avoid sorting $Considered$ nodes and instead
place them into ``buckets'' of
width $\Delta$ based on their tentative values.  If $U(\xbar)$ is the ``smallest'' of
tentative labels and $U(\x)$ is currently in the same bucket, then
even after $\xbar$ is permanently labeled, it cannot affect $U(\x)$ since
$$U(x)  < U(\xbar) + \Delta \leq U(\xbar) + K(x, \xbar).$$
Thus, a typical stage of Dial's method consists of
$Accepting$ everything in the current bucket, recomputing all
not-yet-$Accepted$ nodes adjacent to
those newly $Accepted$, switching them to other buckets if warranted by the
new tentative labels, and then moving on to the next bucket.
Since inserting to and deleting from a bucket
can be performed in $O(1)$ time, the overall computational complexity of Dial's method
becomes $O(M)$.  In addition, while Dijkstra's approach is inherently sequential,
Dial's method is naturally parallelizable.
However, in practice, Dial's method is often less efficient then Dijkstra's because
the constant factor hidden in the former's $O(M)$ asymptotic complexity depends on
$\Delta$ and on average $K_{ij}$ values encountered along optimal paths.
Many other enhancements of the above
label-setting methods are available in the literature (e.g., see
\cite{Bertsekas_NObook}, \cite{Ahuja} and references therein).

\subsubsection{Label-setting for uncertain-horizon problems.}
\label{sss:dijkstra_uncertain}

We now show that a variation of Dijkstra's method is
applicable in the uncertain-horizon case provided
\Assume{1}-\Assume{3} hold.

Part 2 of Theorem \ref{thm:causality} shows the causality
necessary for a variant of Dijkstra's method to correctly
compute the value function.  Some modifications are needed, however.
First of all, the tentative labels should be initially set to $V_i = q_i$.
Secondly, the update formula (\ref{eq:temp_labels_local_Determ}) has to be replaced by
\begin{equation}
\label{eq:temp_labels_local_uncertain}
V(x_i) \; := \; \min \left\{ V(x_i), \left( K(x_i, \xbar) + p q(\xbar) + (1-p) V(\xbar) \right) \right\}.
\end{equation}
Thirdly, there is a question of which nodes should be initially placed
on the $Considered$ list.  (Since this list is maintained using a heap-sort
data structure, this decision will directly impact the efficiency of the method.)
The only way for Dijkstra's to produce a wrong result is to have some node
$\xbar$ $Accepted$ before its optimal neighbor $\hat{x} \in N(\xbar)$.
Suppose $\y$ is $\xbar$'s optimal path becoming motionless at some node $\xtilde$;
i.e.,
$\y = (y_0, \ldots, y_{m-1}, y_m, y_m, y_m, \ldots) \in Y^s(\xbar)$,
where $y_0 = \xbar$, $y_1 = \hat{x}$, and $y_m = \xtilde$.
By Part 2 of Theorem \ref{thm:causality}, $V(\xbar) \geq V(y_1) \geq \ldots V(y_m) = q(y_m)$.
As long as, $y_m$ is initially marked as $Considered$, it is easy to show that
$\xbar$ will not be $Accepted$ prematurely.
So, it is sufficient to initially mark as $Considered$ all
motionless nodes.  Unfortunately, the set $\M$ is not known in advance,
but Part 6 of Theorem \ref{thm:asymptotics_V} can be used to
over-estimate it with $\M^1 = \left\{ x_i \mid \,
q_i \leq K_{ij} +  q_j, \, \forall x_j \in \hat{N}(x_i)\right\}.$
However, the following Lemma shows why it is already sufficient
to mark as $Considered$ a smaller set $Q_l$:

\begin{lemma}
For every $\xbar \in X$ there exists a path
$\y = (y_0, \ldots, y_r)$ such that\\
$y_r \in Q_l$, $y_0 = \xbar$,  and, for all $k=0,\ldots, (r-1)$,\\
\hspace*{1cm} either $y_k$ is a motionless node and $V(y_k)=q(y_k) > V(y_{k+1})$;\\
\hspace*{1cm} or $y_{k+1}$ is an optimal transition from $y_k$\\
\hspace*{13mm} (i.e., $V(y_k) = K(y_k,y_{k+1}) + p q(y_{k+1}) + (1-p) V(y_{k+1})$).
\label{lm:causal_chain}
\end{lemma}
\begin{proof}
Start with an optimal e.m. path for $\xbar$ leading to some motionless node
$\xtilde$. If $\xtilde \not \in Q_l$,
then there exists $x \in N(\xtilde)$ such that
$V(\xtilde) = q(\xtilde) > q(x) \geq V(x)$.
Now concatenate the $x$'s optimal e.m. path and repeat
the process until reaching $Q_l$.
\end{proof}

Using concatenated paths described in Lemma \ref{lm:causal_chain},
it is easy to prove by induction that,
starting with $Q_l$ marked as $Considered$, the order
of acceptance will be also causal (i.e., $\hat{x}$
will be $Accepted$ before $\xbar$).
We summarize the resulting Dijkstra-like method for
uncertain-horizon processes on graphs in Algorithm \ref{alg:D_modified}
(allowing for edge-dependent probabilities of termination $p_{ij}$'s).

\begin{algorithm}[hhhh]
\caption{A Dijkstra-like method for randomly-terminated processes on graphs.}
\label{alg:D_modified}
\noindent
\tt
\hspace*{15mm} start with all nodes marked as $Far$;\\
\hspace*{15mm} set $V_i \; := \; q_i$ for all $\x_i \in X$;\\
\hspace*{15mm} mark all local minima of $q$ as $Considered$;

\noindent
\hspace*{15mm} while ($Considered$ list is not empty) \; \{\\
\hspace*{20mm} let $V_{j} = V(\x_j)$ be the smallest $Considered$ value;\\
\hspace*{20mm} mark $\x_j$ $Accepted$;\\
\hspace*{20mm} for each not-yet-$Accepted$ $\x_i$ such that $\x_j \in N(\x_i)$ \; \{\\
\hspace*{25mm} set $V_i \, := \; \min
\left( V_i,  \;
K_{i j} + p_{i j} q_{j} + (1-p_{i j}) V_{j}
\right);$\\
\hspace*{25mm} if $\x_i$ is $Far$, mark it $Considered$;\\
\hspace*{20mm} \}\\
\hspace*{15mm} \}
\end{algorithm}

The direct procedure of identifying all local minima takes $O(\kappa M)$ operations
and may or may not be advantageous (after all, the algorithm will also
work if the entire $X$ is initially labeled as $Considered$).
However, this pre-processing becomes clearly worthwhile if several different
uncertain-horizon problems are to be solved on the same graph (with the same
terminal cost function $q: X \mapsto \R$) but with different
transition cost functions $K: X \times X \mapsto \R_{+,0}$.

Finally, we note that if $\Delta > 0$, then a similarly modified
Dial's method will also be applicable with buckets of width $\Delta$.

\subsection{Example: optimal idle-time processing.}
\label{ss:optimal_idler}
As a sample application for uncertain-horizon processes,
we consider the task of optimally using the idle-time to
minimize the expected waiting time of the first client.
We model arrival of client requests as a Poisson process with $\lambda$ requests
arriving on average per unit time; i.e., if $T_1$ is the time of the first request,
then $P(T_1 \geq t) = e^{-\lambda t}$ and
$P \left(T_1 \in [t, t+dt] \right) = \lambda e^{-\lambda t} \, dt$.
If $x_i \in X$ is the current state of the system when the first request arrives,
then $q(x_i)$ is the expected time for completing that request.
At every stage of the process we can either decide to stay in the current
state or start a transition to any state $x_j \in N(x_i)$.  Such a transition
will take some time $\tau_{ij} > 0$ and the system will not be capable of responding
to any incoming requests until that transition is completed.
Thus, even if $q(x_j) < q(x_i)$, this has to be balanced against
the expected additional waiting time $(\tau_{ij} - T_1)$ provided $T_1 < \tau_{ij}$,
i.e.,
$$
K_{ij} = \int\limits_0^{\tau_{ij}}  (\tau_{ij} - t) \lambda e^{-\lambda t} \, dt =
\frac{e^{-\lambda \tau_{ij}} - (1 -\lambda \tau_{ij})}{\lambda} \geq 0,
\qquad \forall i \neq j.
$$
If $V_i = V(x_i)$ is the minimal expected wait time starting from the state $x_i$, then
\begin{equation}
V_i \; = \;
\min \left\{ q_i, \;
\min\limits_{x_j \in \hat{N}(x_i)}
\left\{ K_{ij} + p_{ij} q_j + (1-p_{ij}) V_j \right\}
\right\},
\label{eq:discrete_idler}
\end{equation}
where $p_{ij} = P(T_1 < \tau_{ij}) = 1 - e^{-\lambda \tau_{ij}}$.
This system clearly satisfies the assumptions \Assume{1}, \Assume{2}, and \Assume{3}
and the label-setting methods of section \ref{sss:dijkstra_uncertain} are therefore applicable.

One interesting application of this problem is to minimize the expected
wait time of the first emergency caller by moving an ambulance during the idle-time.
In that case, $x$ is the current position of an ambulance and
$q(x)$ is the expected travel time from $x$ to the caller's location.
Suppose $d(x, \xtilde)$ is the minimum travel time through
the graph from $x$ to $\xtilde$ and $\tilde{P}(\xtilde)$ is the probability that
the next call originates from $\xtilde$.  Then
$$
q(x) = \sum\limits_{\xtilde \in X} \tilde{P}(\xtilde) \, d(x, \xtilde).
$$
If $\tilde{P}$ is non-zero for a small fraction of nodes in $X$ only,
then the distances $d(x, \xtilde)$ are best evaluated through a repeated
application of the standard Dijkstra's method.  Otherwise,
Floyd-Warshall algorithm will likely be more efficient \cite{Ahuja}.

Any global minimum of $q$ is obviously the optimal location to ``park''
an ambulance (or to build an ambulance depot).  But if the ambulance
starts at any other location $x_i$, the optimal strategy for moving it
while expecting the call is recovered by using the minimizer
in formula \eqref{eq:discrete_idler}.

\begin{remark}
\label{rem:approx_dynaprog}
After this paper was already submitted, we have found
a somewhat different single-ambulance model
introduced in a recent
technical report \cite{Zhang}, which
also contains an independently developed label-setting algorithm for
the ``no running cost'' subcase (i.e., all $K_{ij} = 0$).

Of course, realistic dynamic fleet management problems
require efficient control of multiple vehicles simultaneously.
The expected response time would then depend on the current position and
status of all vehicles, resulting in the exponential growth
of the state space.  This curse of dimensionality typically precludes
the direct use of the above approach, but a computationally feasible
alternative is provided by the methods of
{\em approximate dynamic programming} \cite{Powell}.
A recent example of the latter applied to the ambulance redeployment problem
can be found in \cite{MRHT}.

Non-iterative algorithms for approximate dynamic programming would be
clearly very useful, but, to the best of our knowledge,
no such methods are currently available.
\end{remark}


\section{Uncertain-horizon processes in continuous state-space.}
\label{s:contin_space}
We now derive the results parallel to those in section \ref{s:general_graph} but
in continuous state-space.  We start with the taxonomy of
``common'' deterministic optimal control problems
on $\domain \subset \R^n$.
We then show the relationship between the uncertain-horizon problems
and time-discounted optimal-stopping problems.
A modified version of Fast Marching Method is then
developed for the latter and illustrated
by a number of numerical examples in section \ref{s:examples}.

\subsection{Common types of optimal control problems in continuous state-space.}
\label{ss:contin_problem_types}

We will assume that $\domain$ is an open bounded subset of $\R^n$.
Suppose $A \in \R^m$ is a compact set of control values,
and the set of admissible controls $\mathcal{A}$ consists of all measurable functions
$\ba: \R \mapsto A$.  The dynamics of the system is defined by
\begin{eqnarray}
\nonumber
\y^{\prime}(s) &=& \fB(\y(s), \ba(s)),\\
\y (0) &=& \x \in \domain,
\label{eq:auton_dynamics}
\end{eqnarray}
where $\y(s)$ is the system state at the time $s$,
$\x$ is the initial system state,
and $\fB: \cdomain \times A \mapsto \R^n$ is the velocity.

We will also assume that the running cost $K: \cdomain \times A \mapsto [K_1, K_2]$ and
the discount rate $\beta \geq 0$ are known.
(The case $\beta=0$ corresponds to no discounting.)


We note that the partial differential equations derived for problems in this
subsection typically do not have classical (smooth) solutions on the entire domain
and the weak solutions are generally not unique.
The theory of {\em viscosity solutions} introduced by Crandall and Lions
\cite{CranLion} is used to overcome this difficulty, picking out
the unique weak solution coinciding with the value function of
the corresponding control problem.
Here we only provide a formal derivation and classification of PDEs;
rigorous proofs and many relevant references can be found in \cite{BardiDolcetta}.

\begin{itemize}
\item
{\em Finite-horizon problems:}\\
Suppose the process runs until the specified time $\TT$.
This allows us to define the total cost of using a control
$\ba(\cdot)$ starting from $\x$
with $T = \TT - t$ seconds left before the termination:
$$
\J(\x, t, \ba(\cdot)) = \int_0^{T} e^{-\beta s} K (\y(s), \ba(s)) \, ds
\, + \, e^{-\beta T} q(\y(T)),
$$
where $t$ is the starting time and $q:\domain \mapsto R$ is the terminal cost.
The value function $u: \domain \times [0, \TT]$ can be defined as
usual:
$
u(\x,t)  =  \inf \J(\x, t, \ba(\cdot)),
$
where the $\inf$ is taken over
$$
\mathcal{A}_0(\x) = \left\{
{\ba(\cdot) \in \mathcal{A}} \, \mid \,
\y(s) \in \cdomain \text{ for all } s \in [0, T] \right\}.
$$
The optimality condition shows that, for any $\tau \in (0,T)$,
$$
u(\x,t) = \inf_{\ba(\cdot)} \left\{
\int_0^{\tau} e^{-\beta s} K (\y(s), \ba(s)) \, ds
\; + \; e^{-\beta \tau} u \left(y(\tau), t+ \tau \right)
\right\}.
$$
Assuming that $u \in C^2$, Taylor-expanding and letting $\tau \to 0$, it is easy
to show that $u$ is a solution of a Hamilton-Jacobi-Bellman PDE:
\begin{eqnarray}
\nonumber
-u_t \, + \, \beta u \, - \,
\min\limits_{\ba \in A}
\left\{
K(\x, \ba) + \nabla u \cdot \fB ( \x, \ba )
\right\}
\, = \, 0,
&& \text{ for } \x \in \domain, t \in [0, \TT);\\
u(\x) \, = \, q(\x), && \text{ for } \x \in \domain, t = \TT.
\label{eq:HJB_general_timeD}
\end{eqnarray}
Here $u_t$ is the derivative with respect to time,
and $\nabla u$ is a vector of spatial derivatives.
This is a {\em terminal value problem} for a hyperbolic non-linear PDE.
The optimal trajectories coincide with the characteristics of this equation.
Since it is time-dependent, the numerical approximation can be
computed very efficiently by {\em time-marching}
(from $t = \TT$ into the past).

\item
{\em Undiscounted exit-time problems:}\\
Here, $\beta = 0$ and the process terminates upon exiting from $\cdomain$; i.e.,
$$
T = T_{\x, \ba} = \inf \{ t \in \R_{+,0} | \y(t) \not \in \cdomain \}
$$
and the terminal cost is defined only on the boundary
($q:\boundary \to \R$).
Since both the cost and dynamics are 
not directly dependent on time, the value function
does not depend on $t$ either and $u(\x)$ is defined by taking $\inf$ over
$\mathcal{A}$.
Similar formal reasoning shows that $u(\x)$ must satisfy
a static Hamilton-Jacobi-Bellman PDE with Dirichlet boundary
conditions
\begin{eqnarray}
\nonumber
-\min\limits_{\ba \in A}
\left\{
K(\x, \ba) + \nabla u \cdot \fB ( \x, \ba )
\right\}
\, = \, 0,
&& \text{ for } \x \in \domain;\\
u \, = \, q, && \text{ for } \x \in \boundary.
\label{eq:HJB_general_exitT}
\end{eqnarray}
If we assume $K_1>0$, this guarantees that wherever $u(\x) < \infty$,
all optimal paths lead to $\boundary$ in finite time.
Thus, all characteristics originate from $\boundary$
and $K_1>0$ implies the monotone growth of $u$ along each
characteristic.  This gives one reasons to hope for causal (label-setting)
numerical methods after the PDE is discretized on a grid.
Such efficient methods were, indeed, constructed for a number of special cases.
In the fully {\em isotropic} case, $K = K(\x)$, $A =S_{n-1}$, and
$\fB = f( \x) \ba$.  Here, a unit vector $\ba$ is our chosen
direction of motion, $f>0$ is the speed of motion, and \eqref{eq:HJB_general_exitT}
reduces to the Eikonal PDE:
\begin{eqnarray}
\nonumber
| \nabla u | f(\x) \, = \, K(\x)
&& \text{ for } \x \in \domain;\\
u \, = \, q, && \text{ for } \x \in \boundary.
\label{eq:Eikonal}
\end{eqnarray}
Two Dijkstra-like methods for discretizations of this equation were introduced
by Tsitsiklis \cite{Tsitsiklis_conference, Tsitsiklis} and Sethian \cite{SethFastMarcLeveSet}.
The former was done in the context of semi-Lagrangian discretizations
for control-theoretic problems; the latter (the Fast Marching Method) was
introduced for upwind finite-difference discretizations of
isotropic front propagation problems.
A detailed discussion of connections between these two approaches can
be found in \cite{SethVlad3}.
The Fast Marching Method was later extended by Sethian and collaborators
to higher order upwind discretizations on grids and unstructured meshes
in $\R^n$ and on manifolds.  All Dijkstra-like methods have
the same computational complexity of $O(M \log M)$,
where $M$ is the number of gridpoints.
Another algorithm introduced by Tsitsiklis
on an 8-neighbor stencil mirrors the logic of Dial's method and results
in $O(M)$ computational complexity (see \cite{Tsitsiklis} and
the generalizations in \cite{VladMSSP}).

A more general anisotropic case is obtained when
$K = K(\x, \ba)$, $A =S_{n-1}$, and
$\fB = f( \x, \ba ) \ba$ (i.e., both the running cost and the speed can depend
on the direction of motion).  Upwind discretizations of the resulting PDE
are generally not causal, making label setting methods inapplicable.
Ordered Upwind Methods circumvent this difficulty by dynamically extending
the stencil just enough to restore the causality, resulting in Dijkstra-like
computational complexity, but with an additional ``anisotropy coefficient''
\cite{SethVlad2, SethVlad3, AltonMitchell2}.

\item
{\em Infinite-horizon problems:}\\
Here, the cost of each control is defined as before, except
that $\beta > 0$, $T$ is always taken to be $+\infty$,
and there is no terminal cost:
$
\J(\x, \ba(\cdot)) = \int_0^{+\infty} e^{-\beta s} K (\y(s), \ba(s)) \, ds.
$
Since we are moving for an infinite time, the value function
$u$ is no longer time-dependent.
(The fact that $u$ is finite follows from $\beta > 0$
and the boundedness of $K$.
To remain in $\domain$, the $\inf$ is taken over $\mathcal{A}_0(\x)$.)
Similar formal reasoning shows that $u(\x)$ must satisfy
a static Hamilton-Jacobi-Bellman PDE
\begin{eqnarray}
\nonumber
\beta u \, - \,
\min\limits_{\ba \in A}
\left\{
K(\x, \ba) + \nabla u \cdot \fB ( \x, \ba )
\right\}
\, = \, 0,
&& \text{ for } \x \in \domain;\\
\label{eq:HJB_general_inf_Hor}
\end{eqnarray}
We note that there is also a well-known version of
this problem, where the process can also be terminated
upon hitting $\boundary$; see \cite{BardiDolcetta}.
In that case, $u$ is defined by taking $\inf$
over $\mathcal{A}$, resulting in Dirichlet boundary
conditions on $\boundary$ interpreted {\em in the viscosity sense}.
However, unlike the undiscounted exit-time problems defined above,
here the optimal trajectories can be periodic inside $\domain$
even if $K_1>0$.
This generally precludes any use of causal numerical methods.

\item
{\em Optimal stopping problems:}\\
Same as the infinite-horizon problem, but with an option
of deciding to terminate the process at any point $T$:
$$
\J(\x, \ba(\cdot), T) = \int_0^{T} e^{-\beta s} K (\y(s), \ba(s)) \, ds
\, + \, e^{-\beta T} q(\y(T)),
$$
where, as in the finite-horizon case, $q:\domain \mapsto R$.
The value function can be defined by taking $\inf$ over all controls
and all termination times:
$$
u(\x) \;  =  \; \inf\limits_{\ba(\cdot) \in \mathcal{A}_0(\x), \, T \geq 0} \,
\J(\x, \ba(\cdot), T).
$$
The direct consequence of this is the fact that
$u(\x) \leq q(\x)$ on $\domain$ and,
wherever $u$ is less than $q$,
it should be formally a solution of \eqref{eq:HJB_general_inf_Hor}.
In this case, $u$ is the viscosity solution of
the following {\em variational inequality of obstacle type} :
\begin{eqnarray}
\label{eq:HJB_opti_stop}
\max \left\{
\, u(\x)-q(\x), \quad
 \beta u \, - \,
\min\limits_{\ba \in A}
\left\{
K(\x, \ba) + \nabla u(\x) \cdot \fB ( \x, \ba )
\right\} \,
\right\}
\, = \, 0,
&& \text{ for } \x \in \domain.
\end{eqnarray}
See \cite{BardiDolcetta} for a rigorous derivation.
Let $\stay = \{ \x \in \domain \mid u(\x) = q(\x) \}$.
Then the PDE \eqref{eq:HJB_general_inf_Hor} is
satisfied on $\move$
and $\B = \fboundary$ forms a {\em free boundary}
for this problem.
It is not hard to show that
the optimal trajectories cannot be periodic, provided
$K(\x, \ba) > \beta q(\x)$ for all $\x$ and $\ba$.
In this case, all characteristics of
this PDE run into $\move$ from $\B$
and the space-marching numerical approach is feasible.
However, since $\B$ is a priori unknown,
this presents an additional challenge for label-setting methods.
(The equivalent properties of uncertain-horizon processes
will be considered in detail in the next subsection.)

We note that the analysis of optimal stopping problems can be reduced to that
of infinite horizon problems if we extend the set of control values
by adding a special ``motionless'' control
$\ba_0$ such that $\fB(\x, \ba_0)=0$ and $K(\x, \ba_0)=\beta q(\x)$
for all $\x \in \domain$.
Similarly, any infinite horizon problem with a motionless control $\ba_0$
available in every state $\x$ can also be viewed as an optimal stopping problem
by setting $q(\x) = K(\x, \ba_0(\x)) / \beta.$
\end{itemize}


\subsection{The value function for uncertain-horizon processes.}
\label{ss:contin_value_func}

We now consider a randomly-terminated state-restricted undiscounted problem.
The dynamics is yet again restricted to $\cdomain$,
the state evolution is described by \eqref{eq:auton_dynamics},
where $\fB(\x, \ba)$ is Lipschitz-continuous in $\x$ and
continuous in $\ba$.
We assume that both $K$ and $q$ are lower semi-continuous functions
satisfying the following:\\

\Bssume{1} $q:\cdomain \mapsto R$ is the terminal cost such that \\
$$q_1 \leq q(\x) \leq q_2, \quad \forall \x \in \cdomain, \qquad
\text{ where } q_1 = \min_{\x} q(\x) \geq 0.$$

\Bssume{2} $K:\cdomain \times A \mapsto R$ is the running cost such that
$0 \leq K_1 \leq K(\x, \ba) \leq K_2, \quad \forall \x \in \cdomain, \ba \in A.$\\

Some of the properties proven below, also rely on a stronger assumption:\\

\Bssume{3} The velocity $\fB: \cdomain \times A \mapsto R^n$ is bounded
(i.e., $|\fB| \leq F_2$) and
the motion in every direction is always possible, i.e.,\\
$\hspace*{2cm}
\forall \x \in \cdomain,$ and all $\bv \in S_{n-1}$, $\quad \exists \ba \in A \text{ s.t. }
\bv \cdot \fB(\x, \ba) = |\fB(\x, \ba)| \geq F_1> 0.$\\

We assume that the terminal event is a result of some Poisson process and
the time till termination has an exponential probability distribution;
i.e., $P(T \geq t) = e^{-\lambda t}$.
Thus, starting from $\x$, the cost
of any individual control $\ba(\cdot) \in \mathcal{A}_0(\x)$ is
\begin{eqnarray}
\label{eq:randomT_trajCost_original}
\J(\x, \ba(\cdot)) & = &
\int_0^{+\infty} \lambda e^{-\lambda t} \left[
\int_0^{t} K (\y(s), \ba(s)) \, ds
\, + \, q(\y(t))
\right] \, dt \\
\label{eq:randomT_trajCost}
& = &
\int_0^{+\infty} e^{-\lambda s} \left[
K (\y(s), \ba(s)) \, + \, \lambda q(\y(s))
\right] \, ds
\end{eqnarray}
where the second equality follows from the Fubini's Theorem.
This shows that any randomly-terminated problem can be converted to
a discounted infinite-horizon problem with $\beta = \lambda$ and
$\tilde{K} (\x, \ba) = K(\x, \ba) + \lambda q(\x).$
This relationship is well-known and has been previously used in
\cite{BoukasHaurieMichael} to treat even more general
random-termination problems
(e.g., with $\lambda = \lambda(\x)$).
The value function can then be defined as usual
$v(\x) = \inf\limits_{\ba(\cdot) \in \mathcal{A}_0(\x)} \J(\x, \ba(\cdot)).$
\begin{remark}
\label{rem:nonnegative_costs}
In view of \eqref{eq:randomT_trajCost},
adding any constants $C_1$ and $C_2$ to functions $K$ and $q$
respectively, will result in adding $C_1/\lambda + C_2$
to the value function $v$.
Thus, the above assumptions about the non-negativity of
$q_1$ and $K_1$ can be made without any loss of generality.
\end{remark}
Assuming $\fB$, $K$ and $q$ are Lipschitz,
the value function $v$ is
a $\cdomain$-{\em constrained viscosity solution}
\cite{BardiDolcetta} of a Hamilton-Jacobi PDE:
\begin{equation}
\label{eq:domain_restricted_sol}
\min_{\ba \in A} \{ \nabla v(\x) \cdot \fB(\x, \ba) + K(\x, \ba)\}
\, + \, \lambda \left( q(\x) - v(\x) \right)
\; = \; 0.
\end{equation}

Before using the fact that $v(\x)$ is the viscosity solution
of the corresponding variational inequality,
we first prove a number of properties of the value function directly.

A standard treatment of state-constrained infinite-horizon
problems shows that, for sufficiently regular costs, dynamics, and $\boundary$,
$v(\x)$ is bounded uniformly continuous.
Using an additional controllability assumption \Bssume{3}
we provide a direct proof of local Lipschitz-continuity of $v(\x)$.
\begin{lemma}
Assume \Bssume{1}-\Bssume{3}.
Then the value function of the randomly-terminated problem
is locally Lipschitz-continuous on $\domain$
with the Lipschitz constant
$L_v = (K_2 + \lambda q_2) / F_1$.
\label{lm:v_Lip}
\end{lemma}
\begin{proof}
Suppose $\x_1, \x_2 \in \domain$ are such that
the straight line segment connecting these two points lies inside
$\domain$.  Suppose $\ba(\cdot)$ is a control and $\y(\cdot)$
is the corresponding trajectory such that
we follow that straight line from $\x_1$ to $\x_2$
and then switch to the optimal control of $\x_2$:
$$
\y(0) = \x_1; \qquad
\y(\tau) = \x_2; \qquad
\J(\x_2, \ba(\cdot\, +\tau) ) = v(\x_2);
$$
$$
\y'(t) \cdot \frac{\x_2 -\x_1}{|\x_2 -\x_1|}  = |\y'(t)| =
|\fB(\y(t), \ba(t))|,
\qquad \text{ for }\forall t \in [0, \tau).
$$
The existence of such $\ba(\cdot)$ and the fact that
$\tau \leq \frac{|\x_2-\x_1|}{F_1}$
follow from the controllability assumption \Bssume{3}.
Then,  from the optimality principle,
\begin{eqnarray*}
v(\x_1) &\leq& \J(\x_1, \ba(\cdot)) =
\int_0^{\tau} e^{-\lambda s} \left[
K (\y(s), \ba(s)) + \lambda q(\y(s))
\right] \, ds \, + \, e^{-\lambda \tau} v(\x_2)\\
& \leq &
\tau  (K_2 + \lambda q_2) + v(\x_2)
\; \leq \;
\frac{|\x_2-\x_1|}{F_1}(K_2 + \lambda q_2) + v(\x_2)
= |\x_2-\x_1| L_v  + v(\x_2).
\end{eqnarray*}
To complete the proof, the above argument can be repeated
reversing the roles of $\x_1$ and $\x_2$.
We note that this proof does not use any additional assumptions on
the regularity of $K$ or $q$.
\end{proof}


We now list three essential assumptions that
provide a continuous equivalent of \Assume{1}-\Assume{3},
which defined the discrete randomly-terminated problem
in section \ref{ss:graph_uncertain_horizon}.
We assume the existence of a special control $\ba_0 \in A$ such that\\

\AssumePrime{1}
$ \hspace*{2cm} \fB(\x, \ba_0) = 0, \qquad \forall \x \in \cdomain.$\\

\AssumePrime{2}
$\hspace*{2cm} K(\x, \ba_0) = 0, \qquad \forall \x \in \cdomain.$\\

\noindent
We will also use $\hat{A} = A \backslash\{\ba_0\}$ for the set of all other control values.\\

\AssumePrime{3}
$\hspace*{2cm}
\widebar{K}_1 =
\inf\limits_{\x \in \domain, \, \ba \in \hat{A}} K(\x, \ba) \, \geq \,0.$

\vspace*{2mm}
Assumptions \AssumePrime{1}-\AssumePrime{3} allow us to restate
\eqref{eq:domain_restricted_sol} as an
{\em optimal stopping problem}
\begin{equation}
\max \left\{
\, v(\x)-q(\x), \quad
 \lambda \left( v(\x) - q(\x) \right) \, - \,
\min\limits_{\ba \in \hat{A}}
\left\{
K(\x, \ba) + \nabla v(\x) \cdot \fB ( \x, \ba )
\right\} \,
\right\}
\, = \, 0;
\label{eq:obstacle_form}
\end{equation}
or, in the isotropic case (when $A = S_{n-1} \cup \{\ba_0\}$,
and  $\fB ( \x, \ba ) = f(\x) \ba, \; K(\x, \ba) = K(\x), \; \forall \ba \neq \ba_0$),
\begin{equation}
\max \left\{
\, v(\x)-q(\x), \;
 \lambda \left( v(\x) - q(\x) \right) \, - \, K(\x)  \, + \, |\nabla v(\x)| f(\x)
\,
\right\}
\, = \, 0.
\label{eq:obstacle_form_iso}
\end{equation}

Using the notation $[z]^{-} = \min (z, 0)$,
we can rewrite this variational inequality as follows:
\begin{equation}
v(\x) \; = \; q(\x) \; + \; \frac{1}{\lambda}
\left[ \min\limits_{\ba \in \hat{A}}
\left\{
K(\x, \ba) + \nabla v(\x) \cdot \fB ( \x, \ba )
\right\}
\right]^{-},
\label{eq:general_var_ineq_1}
\end{equation}
or, in the isotropic case,
\begin{equation}
v(\x) \; = \; q(\x) \; + \; \frac{1}{\lambda}
\left[ K(\x) - f(\x) | \nabla v(\x) | \right]^{-}.
\label{eq:Eikonal_var_ineq}
\end{equation}

The PDE holds (with the $[\cdot]^-$ omitted) on
the domain $\cdomain \backslash \M$.
In section \ref{ss:FMM_modified} we discuss the discretized version
of the isotropic variational inequality \eqref{eq:Eikonal_var_ineq}
and a label-setting algorithm for solving it efficiently.
We note that several properties of the value function
(e.g., Part 5 of Theorem \ref{thm:v_causality})
can also be obtained directly from the characteristic
ODEs of \eqref{eq:general_var_ineq_1} using
stronger regularity assumptions on the dynamics and cost.
However, our proofs provided below are more general.
Examples studied in sections \ref{ss:contin_idle} and \ref{ss:maze}
use discontinuous $f$ and $K$ to test our numerical method;
additional examples (using discontinuous $q$) were omitted
to save space.
A detailed theoretical discussion of viscosity solutions to HJB PDEs with
discontinuous Lagrangian can be found in \cite{Soravia_DiscontinLagrangian}.

The rest of this subsection is based on direct applications of
the control-theoretic interpretation rather than on the theory
of viscosity solutions.
We will refer to a control $\ba(\cdot)$
as {\em eventually motionless} (or {\em e.m.}) if $\exists T$ s.t.
$\ba(t) =\ba_0, \, \forall t \geq T.$  We will say that
the corresponding e.m. trajectory $\y(\cdot)$ becomes motionless at the point $\y(T)$.
Given a fixed control $\ba(\cdot)$, we will further use a notation
$$
\omega(t) = \int_0^{t} K (\y(s), \ba(s)) \, ds
\, + \, q(\y(t)),
$$
for the deterministic cost of it if the termination happens at the time $t$.
We note that, for an e.m. control, formula \eqref{eq:randomT_trajCost_original}
can be rewritten as
\begin{equation}
\label{eq:traj_cost_terminal}
\J(\x, \ba(\cdot)) \; = \;
\int_0^{T} \lambda e^{-\lambda t} \omega(t) \, dt \; + \; e^{-\lambda T} \omega(T).
\end{equation}
We will also define a ``motionless'' set
$\M = \{ \x \mid v(\x) = \J(\x, \ba_0)  = q(\x) \}.$
We note that, if the value function is continuous, then the set $\M$
is closed (since $(v-q)$ is upper semi-continuous).


The following two theorems list a number of properties of the value function
and of optimal controls/trajectories.

\begin{thm}
Assume \Bssume{1},\Bssume{2}, and \AssumePrime{1}-\AssumePrime{3}.
Then, for all $\x \in \domain$,
\begin{enumerate}

\item
$v(\x) \in [q_1,  q(\x)]$.

\item
$q(\x)=q_1$ implies $\x \in \M$.

\item
$\forall \x \in \domain, \varepsilon > 0, \, \exists$
an $\varepsilon$-suboptimal e.m. control $\ba^{\varepsilon}(\cdot).$

\item
If $\widebar{K}_1 > 0$, then every optimal control is e.m.

\item
If $\y^*(t)$ is an optimal trajectory,
then $v(\y^*(t))$ is a monotone non-increasing function of $t$
(and monotone decreasing up to its motionless point if $\widebar{K}_1 > 0$).

\item
If \Bssume{3} holds, $\widebar{K}_1 > 0$, and $\y^*(t)$ is an optimal trajectory
for $\x = \y^*(0)$, then
$$|\y^*(t) - \x| \geq t \frac{\widebar{K}_1 F_1}{K_2 + \lambda q_2},$$
provided the straight line from $\x$ to $\y^*(t)$ lies within $\cdomain$.

\item
Let $D(\x_1, \x_2)$ be the minimum distance
from $\x_1$ to $\x_2$ (minimized over all allowable trajectories in $\cdomain$)
and $D = \sup\limits_{\x_1, \x_2 \in \cdomain} D(\x_1, \x_2)$.
If \Bssume{3} also holds, then $v(\x) \leq q_1 + D L_v$.

\end{enumerate}
\label{thm:v_causality}
\end{thm}

\begin{proof} \mbox{ }\\

\vspace*{-3mm}
\begin{enumerate}
\item
\AssumePrime{1}, and \AssumePrime{2} ensure that
$v(\x) \leq \int_0^{+\infty} e^{-\lambda s}  \lambda q(\x) \, ds  = q(\x) $,\\
while \AssumePrime{3} shows that for every control
$\J(\x, \ba(\cdot))  \geq
\int_0^{+\infty} e^{-\lambda s} \lambda q(\y(s)) \, ds \geq q_1.$
\item
In particular, if $q_1 = q(\x)$ then $v(\x) = q(\x)$ and $\ba^*(t) = \ba_0$ is the optimal control.
\item
Suppose $\ba_1(\cdot)$ is a (possibly non-e.m.) $\frac{\varepsilon}{2}$-suboptimal
control and $\y_1(\cdot)$ is the corresponding trajectory starting from $\x$; i.e.,
$$
\J(\x, \ba_1(\cdot)) \; = \;
\int_0^{+\infty} e^{-\lambda s} \left[
K (\y_1(s), \ba_1(s)) \, + \, \lambda q(\y_1(s))
\right] \, ds  \; \leq \; v(\x) + \frac{\varepsilon}{2}.
$$
Define a control $\ba_2(t) = \ba_1(t)$ for $t < \tau$ and
$\ba_2(t) = \ba_0$ after that.  The corresponding trajectory $\y_2$ is
a truncated version of $\y_1$.
\begin{eqnarray*}
&&\J(\x, \ba_2(\cdot)) =
\int_0^{\tau} e^{-\lambda s} \left[
K (\y_1(s), \ba_1(s)) \, + \, \lambda q(\y_1(s))
\right] \, ds + \int_{\tau}^{+\infty} e^{-\lambda s} \lambda q(\y_1(\tau)) \, ds\\
& = &
\int_0^{+\infty} e^{-\lambda s} \left[
K (\y_1(s), \ba_1(s)) \, + \, \lambda q(\y_1(s))
\right] \, ds + \int_{\tau}^{+\infty} e^{-\lambda s}
\left[\lambda q(\y_1(\tau)) - K (\y_1(s), \ba_1(s)) \, - \, \lambda q(\y_1(s)) \right] \, ds\\
& \leq &
\J(\x, \ba_1(\cdot)) \, + \,
\left[ q_2 - q_1 - \frac{K_1}{\lambda} \right] e^{-\lambda \tau}.
\end{eqnarray*}
Thus, for all sufficiently large $\tau$,
$$\J(\x, \ba_2(\cdot)) \leq \J(\x, \ba_1(\cdot)) + \frac{\varepsilon}{2}
\leq v(\x) + \varepsilon; \hspace*{2cm}
\text{ i.e., $\ba_2(\cdot)$ is $\varepsilon$-suboptimal.}
$$
\item
Suppose $\ba^*(t)$ is a non-e.m. optimal control starting from $\x$ and
$\y^*(t)$ is the corresponding trajectory.
Let $\tilde{q}_1 = \inf\limits_{t>0} q(\y^*(t)).$  Then $\exists \tau > 0$ such
that $ q(\y^*(\tau)) - \tilde{q}_1 < \widebar{K}_1 / \lambda$.
Then
$$
\int_{\tau}^{+\infty} e^{-\lambda s} \left[
K (\y^*(s), \ba^*(s)) \, + \, \lambda q(\y^*(s))
\right] \, ds
\; \geq \;
e^{-\lambda \tau} \left[
\widebar{K}_1 / \lambda \, + \, \tilde{q}_1
\right]
\; > \;
e^{-\lambda \tau} q(\y^*(\tau));
$$
i.e., making
that trajectory motionless
at $\y^*(\tau)$ would result in a lower cost,
contradicting the optimality of $\ba^*(t)$.

\item
Suppose $\ba^*(\cdot)$ is an optimal control starting from $\x$,
$\y^*(t)$ is the corresponding optimal trajectory,
and define $z(t) = v(\y^*(t))$.
By the optimality principle,
$$
z(0) =
\int_0^{t} e^{-\lambda s} \left[
K (\y^*(s), \ba^*(s)) + \lambda q(\y^*(s))
\right] \, ds \, + \, e^{-\lambda t} z(t).
$$
By Lebesgue's differentiation theorem, $z(t)$ is differentiable for
almost all $t>0$ and
$$
0 = e^{-\lambda t} \left[
K (\y^*(t), \ba^*(t)) + \lambda q(\y^*(t))
\right]
\, + \, e^{-\lambda t} z'(t)
\, - \, \lambda e^{-\lambda t} z(t);
$$
Hence, $z'(t) = \lambda
\left[ z(t) - q(\y^*(t)) \right] - K (\y^*(t), \ba^*(t))$
for almost all $t>0$.
Since $z(t) = v(\y^*(t)) \leq q(\y^*(t))$, we have $z'(t) \leq 0$;
moreover, this inequality becomes strict if $\widebar{K}_1 > 0$ and
$\ba^*(t) \neq \ba_0$.

We note that a similar argument combined with
the optimality principle shows that, even for a non-optimal
control $\ba(\cdot)$, the value function along
the corresponding trajectory $z(t) = v(\y(t))$
will satisfy
$z'(t) \geq \lambda
\left[ z(t) - q(\y(t)) \right] - K (\y(t), \ba(t))$
for almost every $t>0$.


\item
As shown above, $z'(t) \leq - \widebar{K}_1$;
so, $v(\x) - v(\y^*(t)) \geq t \widebar{K}_1$.
On the other hand, by Lemma \ref{lm:v_Lip},
$v(\x) - v(\y^*(t)) =  |v(\x) - v(\y^*(t))| \leq L_v |\x - \y^*(t)|.$
Thus,
$|\y^*(t) - \x| \geq t \frac{\widebar{K}_1 F_1}{K_2 + \lambda q_2}.$

\item
Since $v$ is locally Lipschitz by Lemma \ref{lm:v_Lip}, this
bound is obtained by following the trajectory from $\x$ to
$\x_0^* = \argmin\limits_{\x_0 \in \cdomain} q(\x_0)$.
\end{enumerate}
\end{proof}







In the rest of this subsection we will use superscripts
to indicate the dependence on $\lambda$ of the expected control cost
$\J^{\lambda}(\x, \ba(\cdot))$,
the value function $v^{\lambda}(\x)$ and
the motionless set $\M^{\lambda} = \{\x \mid v^{\lambda}(\x) = q(\x) \}$.
Wherever this superscript is omitted,
the properties hold for any fixed $\lambda > 0$.
We also introduce $v^0$ as the solution of a variational inequality:
\begin{equation}
\label{eq:DP_v0}
\max \left\{
\, v^0(\x)-q(\x), \;
 - \min\limits_{\ba \in A}
\left\{
K(\x, \ba) + \nabla v^0(\x) \cdot \fB ( \x, \ba )
\right\} \,
\right\}
\, = \, 0.
\end{equation}

We see that \eqref{eq:DP_v0} is a variant of \eqref{eq:HJB_opti_stop}
with $\beta = 0$.   Intuitively, $v^0$ can be interpreted as the value function
for the limiting case,
in which the termination is so unlikely that we are guaranteed to reach
any $\x \in \cdomain$ of our choice before it occurs.  Due to \AssumePrime{2},
\AssumePrime{3} we can instead minimize over $\hat{A}$
without changing $v^0$.

Let  $\D_{+}^{\ba} q(\x)$
and
${\mathcal{H}}_{+}^{\ba} q(\x)$
be respectively
the lower Dini and the lower Hadamard derivatives
of $q$ along the vector $\fB ( \x, \ba )$; i.e.,
$$
\D_{+}^{\ba} q(\x) \; = \;
\liminf\limits_{t \to 0^+}
\frac{q \left( \x + t \fB ( \x, \ba ) \right) - q(\x)}{t};
\qquad
{\cal{H}}_{+}^{\ba} q(\x) \; = \;
\liminf\limits_{\begin{array}{c} t \to 0^+\\ \bb \to \fB ( \x, \ba ) \end{array}}
\frac{q \left( \x + t \bb \right) - q(\x)}{t};
$$
if $q$ is differentiable, then
$\D_{+}^{\ba} q(\x) = {\cal{H}}_{+}^{\ba} q(\x) = \nabla q(\x) \cdot \fB ( \x, \ba )$.
We also define two sets
$$
\M^{\infty}
 \, = \,
\left\{\x \mid \inf\limits_{\ba \in \hat{A}}
\left\{
K(\x, \ba) + \D_{+}^{\ba} q(\x)
\right\} \geq 0 \right\}
\quad
\text{ and }
\quad
\M^{\infty}_0
 \, = \,
\left\{\x \mid \min\limits_{\ba \in \hat{A}}
\left\{
K(\x, \ba) + {\cal{H}}_{+}^{\ba} q(\x)
\right\} > 0 \right\}.
$$
The minimum in the above definition of
$\M^{\infty}_0$ is attained since the lower Hadamard derivatives
are lower semi-continuous functions of the direction.
Moreover, it is easy to show that for every $\x \in \M^{\infty}_0$
there exists $\epsilon >0$ and $\tau > 0$ such that
for all $t \in (0, \tau]$ and all $\ba \in \hat{A}$
\begin{equation}
\label{eq:q_around_Minf_0}
K(\x, \ba) +
\frac{q \left( \x + t \fB ( \x, \ba ) \right) - q(\x)}{t}
\geq \epsilon > 0.
\end{equation}
\begin{remark}
\label{rem:inf_motionless_sets}
The set $\domain \backslash \M^{\infty}$ consists of points,
starting from which it is never optimal to stay
in place regardless of how high $\lambda$ is.
However, the points in $\M^{\infty} \backslash \M^{\infty}_0$
may also remain outside of $\M^{\lambda}$ for any finite $\lambda$.
E.g., if $K \equiv 0$ and $q$ is smooth,
then every critical point of $q$ will be in $\M^{\infty}$,
including all strict local maxima of $q$,
starting from which it is clearly better to move regardless of $\lambda$,
and all strict local minima of $q$, which become motionless
when $\lambda$ is sufficiently large.
Below we show that
all point in $\M^{\infty}_0$ also have the latter property,
at least for isotropic cost/dynamics with $\widebar{K}_1 > 0$.
Interestingly, the isotropy and $\widebar{K}_1 > 0$ also imply that
all strict local maxima of a smooth $q$ will also
lie in $\M^{\infty}_0$.  This might seem somewhat counterintuitive,
but there is no contradiction: for large $\lambda$, the process
will likely terminate close to the starting point,
and the net decrease in $q$ will be insufficient
to compensate for the accumulated running cost.
\end{remark}

To simplify the notation, we will refer to a control
$\ba^*(t)$  (and the corresponding trajectory $\y^*(t)$)
as {\em $\lambda$-optimal} if
$\J^{\lambda}(\x, \ba^*(\cdot))=v^{\lambda}(\x).$
\begin{thm}
Bounds and asymptotic behavior of $v^{\lambda}$ and $\M^{\lambda}$
assuming \Bssume{1},\Bssume{2}, and \AssumePrime{1}-\AssumePrime{3}.
Suppose that an e.m. $\lambda$-optimal control exists for every
starting location $\x \in \domain$ and $\forall \lambda \geq 0$.
\begin{enumerate}

\item
If $\y^*(t)$ is an optimal trajectory starting from $\x$ and becoming motionless
at some $\x_0 = \y^*(t_0)$, then
$\omega(t_0) \leq \omega(t)$ and $q(\x_0) \leq q(\y^*(t))$ for all $t \in [0, t_0)$
(and the second inequality is strict if $\widebar{K}_1 > 0$).

\item
$v^{\lambda}(\x) \in [v^0(\x), \, q(\x)]$ for all $\lambda > 0$ and all $\x \in \cdomain$.

\item
$0 \leq \lambda_1 \leq \lambda_2 \qquad \imply \qquad
v^{\lambda_1}(\x) \leq v^{\lambda_2}(\x),
\quad \forall \x \in \cdomain.$

\item
$v^{\lambda} \to q$ pointwise as $\lambda \to +\infty$.

\item
$v^{\lambda} \to v^0$ pointwise as $\lambda \to 0$.

\item
If $\lambda_1 \leq \lambda_2$, then $\M^{\lambda_1} \subset \M^{\lambda_2}. \,$
In particular, $\M^0 \; \subset\;  \M^{\lambda} \; \subset \; \M^{\infty}, \,$
for all $\lambda > 0$.


\item
Suppose the cost and dynamics are isotropic as in equation
\eqref{eq:obstacle_form_iso} and $\widebar{K}_1 > 0$.
Then, for every $\x \in \M^{\infty}_0$, there exists a sufficiently large
$\lambda >0$, such that $\x \in \M^{\lambda}$.

\item
For every $\x \not \in \M^0$, there exists a sufficiently small
$\lambda >0$, such that $\x \not \in \M^{\lambda}$.

\item
We will denote the sets of local minima,
strict local minima and global minima of $q$ in $\cdomain$
as $Q_l$, $Q_l^s$ and $Q_g$ respectively.
Then $Q_g \subset \M^0$, and $Q_l \subset \M^{\infty}$.
Moreover, if \Bssume{3} also holds and
$K \equiv 0,$ then $\M^{\infty}_0 \subset Q_l^s.$
\end{enumerate}
\label{thm:asymptotics_v}
\end{thm}

\begin{proof}
\begin{enumerate}
\item
Suppose $t_1 = \argmin_{t \in [0, t_0]} \omega(t)$ and
$\x_1 = \y^*(t_1)$.  If $\omega(t_0) > \omega(t_1)$,
then formula \eqref{eq:randomT_trajCost_original}
shows that the same trajectory but made motionless earlier (at $\x_1$)
would have a lower cost, contradicting the optimality of $\y^*(\cdot)$.
Thus, $\omega(t_0) \leq \omega(t)$ for all $t \in [0, t_0)$.
As a result,
$$
\omega(t) = \int_0^{t} K (\y^*(s), \ba^*(s)) \, ds \, + \, q(\y^*(t))
\geq \omega(t_0) \geq \int_0^{t} K (\y^*(s), \ba^*(s)) \, ds \, + \, q(\x_0),
$$
where the last inequality uses \AssumePrime{3} and becomes strict
if $\widebar{K}_1 > 0$.  Hence, $q(\x_0) \leq q(\y^*(t))$.\\
We emphasize that the above argument {\em does not} imply that
$\omega(t)$ is monotone decreasing along optimal trajectories.
In fact, for $K=0$, it is easy to see that any optimal trajectory
starting from a local minimum of $q$ (but outside $\M$)
would provide a counter-example.
\item
Suppose that $\ba^*(t)$ is a $\lambda$-optimal control starting from $\x$ and
becoming motionless at the time $t_0$.
From the formula \eqref{eq:traj_cost_terminal} and using Part 1 of the current Theorem,
$
v^{\lambda}(\x) = \J^{\lambda} (\x, \ba^*(\cdot))
= \int_0^{t_0} \lambda e^{-\lambda t} \omega(t) \, dt \; + \; e^{-\lambda t_0} \omega(t_0)
\geq
\omega(t_0) \geq v^{0}(\x),$
where the last inequality reflects the interpretation of $v^{0}$ as the value function
of the deterministic optimal stopping problem.
\item
Suppose that $\ba(t)$ is an e.m. $\lambda_2$-optimal control
and $\y(t)$ is the corresponding trajectory starting
from $\x$.
Define $z_2(t) = v^{\lambda_2}(\y(t))$ and $z_1(t) = v^{\lambda_1}(\y(t))$.
Since this trajectory is $\lambda_2$-optimal,
part 5 of theorem \ref{thm:v_causality}
shows that, for almost all $t$,
$$
z_2'(t) \; = \; \lambda_2
\left[ z_2(t) - q(\y(t)) \right] - K (\y(t), \ba(t))
\; \leq \; \lambda_1
\left[ z_2(t) - q(\y(t)) \right] - K (\y(t), \ba(t)),
$$
where the inequality uses the fact that $z_2(t) \leq q(\y(t)).$
Since the same trajectory is not necessarily $\lambda_1$-optimal,
we have
$$
z_1'(t) \; \geq \; \lambda_1
\left[ z_1(t) - q(\y(t)) \right] - K (\y(t), \ba(t)).
$$
Subtracting these inequalities, and defining $\theta(t) = z_2(t) - z_1(t)$,
we see that
$\theta(t)' \leq \lambda_1 \theta(t)$.
On the other hand, when
the trajectory becomes motionless at the time $t_0 >0$, we know that
$v^{\lambda_2}(\y(t_0)) = q(\y(t_0)) \geq v^{\lambda_1}(\y(t_0))$;
i.e., $\theta(t_0) \geq 0$.  Taken together with the above differential inequality
this implies $\theta(0) \geq 0$; i.e.,  $v^{\lambda_2}(\x) \geq v^{\lambda_1}(\x)$.
\item
Note that $\omega(t)$ is lower semicontinuous for every control $\ba(\cdot)$.
Since $q$ is lower semicontinuous, for every $\x \in \domain$ and every $\epsilon>0$
there exists $\delta$ such that $|\x-\xBtilde|<\delta \imply q(\xBtilde) > q(\x) - \epsilon.$
Choosing $\tau < \delta / \|\fB\|_{\infty}$, we can bound from below the cost of
every control
$$
\J^{\lambda}(\x, \ba(\cdot)) \geq \int_0^\tau \omega(s) \lambda e^{ -\lambda s} \, ds
\geq \int_0^\tau q(\y(s)) \lambda e^{ -\lambda s} \, ds
\geq \left( q(\x) - \epsilon \right) \left( 1-e^{-\lambda \tau} \right).
$$
Thus, $q(\x) \geq v^{\lambda}(\x) \geq \left( q(\x) - 2 \epsilon \right)$
for all sufficiently large $\lambda$'s.  Since this argument works for any
$\epsilon>0$, $v^{\lambda}(\x) \to q(\x)$
as $\lambda \to +\infty$.

\item
Suppose a $0$-optimal control $\ba(\cdot)$ becomes motionless at some time $T$;
i.e., $v^0(\x) = J^{0}(\x, \ba(\cdot)) =
\int_0^T K (\y(s), \ba(s)) \, ds \, + \, q(\y(T)).$
Then, from formula \eqref{eq:traj_cost_terminal},
$\lim\limits_{\lambda \to 0} \J^{\lambda}(\x, \ba(\cdot)) = \omega(T) = v^0(\x).$
\item
If $\x \in \M^{\lambda_1}$, then $q(\x) = v^{\lambda_1}(\x) \leq v^{\lambda_2}(\x) \leq q(\x)$.
Thus, $\x \in \M^{\lambda_2}.$
\item
Suppose $\x \in \M^{\infty}_0$ is fixed.
To prove that $\x$ becomes motionless for some finite $\lambda$,
it is not enough to show that for every fixed control $\ba(\cdot)$
we can choose large enough $\lambda$ to guarantee $\J^{\lambda}(\x, \ba(\cdot)) \geq q(\x).$
We need to show that some finite $\lambda$ guarantees this inequality for {\bf all} $\ba(\cdot)$.

For isotropic dynamics and cost, the fact that $\x \in \M^{\infty}_0$ guarantees that
$\exists \delta > 0$ such that $|\y-\x| \leq \delta$ implies
$$K(\x) + \frac{q(\y)-q(\x)}{t_{\blds{y}}} \geq \epsilon > 0,$$
where $t_{\blds{y}}$ denotes the minimum time needed to reach from $\x$ to $\y$.
Without loss of generality, we can assume that $\epsilon < \widebar{K}_1$.
Let $\ba(\cdot)$ be an arbitrary control with the corresponding trajectory
$\y(t)$ starting from $\x$.  We also choose a small enough
$\tau$ to ensure that
$|\y(t)-\x| \leq \delta$
and $K(\y(t)) \geq  K(\x) - \epsilon/2$
for all $t \leq \tau$.  Then
$$
\J^{\lambda}(\x, \ba(\cdot)) \geq \int_0^\tau \omega(s) \lambda e^{ -\lambda s} \, ds
\geq
\int_0^\tau \left[ \left( K(\x) - \frac{\epsilon}{2} \right)s \, + \, q(\y(s)) \right] \lambda e^{ -\lambda s} \, ds.
$$
Using the fact that $t_{\mbox{\scriptsize {\y}(s)}} \leq s$, we note that
$$
q(\y(s)) \; \geq  \; q(\x) + t_{\mbox{\scriptsize {\y}(s)}} \left( \epsilon - K(\x) \right)
\; \geq \; q(\x) + s \left( \epsilon - K(\x) \right).
$$
Combining this with the above,
\begin{eqnarray*}
\J^{\lambda}(\x, \ba(\cdot)) \geq
\int_0^\tau \left[ \frac{s \epsilon}{2}  \, + \, q(\x) \right] \lambda e^{ -\lambda s} \, ds
& = &
(1-e^{-\lambda \tau}) q(\x) + \frac{\epsilon}{2 \lambda}
\left[1 - e^{-\lambda \tau} - \lambda \tau e^{-\lambda \tau} \right]\\
& = &
q(x) + \frac{\epsilon}{2 \lambda}
\left[1 - e^{-\lambda \tau} - \lambda \tau e^{-\lambda \tau}
- 2 \lambda \epsilon^{-1} q(\x) e^{-\lambda \tau}
\right].
\end{eqnarray*}
To complete the proof, we note that, for large enough $\lambda$,
the expression in the last square brackets is strictly positive,
and this inequality holds for all controls $\ba(\cdot)$.

\item
Follows from part 5.
\item
Suppose $\x \in Q_g$ and $\ba(\cdot)$ is its $0$-optimal control
that becomes motionless at some time $T$.  Then
$v^0(\x) = \omega(T) \geq q(\y(T)) \geq q(\x)$; hence, $\x \in \M^0$.
If $\x \in Q_l$, then $\D_{+}^{\ba} q(\x) \geq 0, \forall \ba \in A$
and $\x \in \M^{\infty}.$
Finally, if $K \equiv 0$, then \Bssume{3} implies that all
lower Hadamard derivatives of $q$ are positive at every
$\x \in \M^{\infty}_0$, making $x$ a strict minimum.
\end{enumerate}
\end{proof}

\noindent
We note that the conditions in Theorem \ref{thm:asymptotics_v} can be further relaxed
by rewriting the proofs in terms of $\epsilon$-suboptimal e.m. trajectories
(see Part 3 of Theorem \ref{thm:v_causality}).

\subsection{Upwind discretization and a modified Fast Marching Method.}
\label{ss:FMM_modified}

We consider a first-order upwind finite differences discretization
of the isotropic variational inequality \eqref{eq:Eikonal_var_ineq}
and introduce a modified version of Fast Marching Method applicable to it.
We note that Ordered Upwind Methods \cite{SethVlad2, SethVlad3, AltonMitchell2}
can be
similarly modified to handle anisotropic randomly-terminated problems.

We assume that \eqref{eq:Eikonal_var_ineq} is discretized
on a uniform Cartesian grid\footnote{
For the sake of notational simplicity, we describe the numerical method in
$R^2$; the generalization for $n>2$ is straightforward.}
with $M$ gridpoints and the value function $v(x,y)$ is approximated by a grid-function $V$:
$$
\x_{i,j} = (x_i, y_j); \qquad
x_{i\pm1}  = x_i \pm h; \qquad
y_{j\pm1}  = y_j \pm h; \qquad
v(x_i, y_j) \approx V(x_i, y_j) = V_{i,j}.
$$
We also define the set of neighboring gridpoints and the set of neighboring values
$$
N(\x_{i,j})=N_{i,j}=\{ \x_{i+1,j}, \x_{i,j+1}, \x_{i-1,j}, \x_{i,j-1}\}; \qquad
NV_{i,j}=\{ V_{i+1,j}, V_{i,j+1}, V_{i-1,j}, V_{i,j-1}\}.
$$

Our discretization uses two one-sided, first-order accurate
approximations for each partial derivative; i.e.,
$$
v_x(x_i, y_j) \approx D^{\pm x}_{ij} V =
\frac{ V_{i \pm 1, j} - V_{i,j} }{ \pm h}; \qquad
v_y(x_i, y_j) \approx D^{\pm y}_{ij} V =
\frac{ V_{i, j \pm 1} - V_{i,j} }{ \pm h}.
$$
An {\em upwind discretization} of \eqref{eq:Eikonal_var_ineq}
at a gridpoint $(x_i, y_j)$ is obtained as follows:
\begin{equation}
 V_{i,j} \; = \;  q_{i,j} + \frac{1}{\lambda} \left[
K_{i,j}
 - f_{i,j} \sqrt{
\left( \max \left( D^{-x}_{ij}V, \, -D^{+x}_{ij}V, \, 0 \right)\right)^{2}
\; + \;
\left( \max \left( D^{-y}_{ij}V, \, -D^{+y}_{ij}V, \, 0 \right)\right)^{2}
}
\right]^{-}.
\label{eq:DP_cUHP}
\end{equation}
If all the $NV_{i,j}$ values are already known, \eqref{eq:DP_cUHP}
has to be solved to obtain $V_{i,j}$.  The latter task is significantly
simplified once we realize that the equation can be solved on a
quadrant-by-quadrant basis.  This procedure is described in detail in
Appendix B (section \ref{s:Appendix_B}) and there we also prove that the
finite-difference discretization \eqref{eq:DP_cUHP} can be
obtained from the Kuhn-Tucker optimality conditions for
a suitable semi-Lagrangian discretization.

However, the values in $NV_{i,j}$ are not a priori known
and equation \eqref{eq:DP_cUHP} has to hold at every gridpoint,
resulting in a system of $M$ coupled non-linear equations.
This system can be solved iteratively, but that approach is
unnecessarily inefficient.  Indeed, \eqref{eq:DP_cUHP}
is related to the upwind scheme used by Rouy and Tourin \cite{RouyTour}
for the Eikonal equation, which Sethian later showed to possess causal properties,
yielding the non-iterative Fast Marching Method \cite{SethFastMarcLeveSet}.
Below we provide an extension of that method to our obstacle problem.

\begin{remark}
\label{rem:discretization_properties}
If $V$ satisfies equation \eqref{eq:DP_cUHP}, it is easy to show that\\
1. $V_{i,j} \leq q_{i,j};$ i.e., the discretized version is also an obstacle problem.\\
2. $V_{i,j}$ is a non-decreasing function of all values in
$NV_{i,j}.$
(This monotonicity, along with the consistency of the discretization
can be used to show the convergence of $V$ to $v$ as $h \to 0$;
see \cite{BarlesSouganidis, ObermanSINUM}.)\\
3. Suppose $\hat{V} \in NV_{i,j}$.
Then, either $\hat{V} < V_{i,j}$ or any increase in $\hat{V}$ will not affect
$V_{i,j}$.
In other words, the value of $V_{i,j}$ depends only on its smaller neighbors,
which makes the label-setting method summarized in Algorithm
\ref{alg:FMM_modified} applicable.\\
\end{remark}

As with Dijkstra's method, the values $V$ at all gridpoints are at first temporary,
equation \eqref{eq:DP_cUHP} is used to update these temporary values,
and the values become permanent when the corresponding gridpoints are $Accepted$.
Due to causality of this discretization,
by the time the algorithm terminates, \eqref{eq:DP_cUHP} holds at all gridpoints.
As in the original Fast Marching Method for the Eikonal equation,
the computational cost of our algorithm is $O(M \log M)$, where the $\log M$
term stems from implementing a sorted list of $Considered$ gridpoints
using a heap-sort data structure.


\begin{algorithm}[hhhh]
\caption{A modified Fast Marching Method for randomly-terminated isotropic problems.}
\label{alg:FMM_modified}
\noindent
\tt
\hspace*{15mm} start with all gridpoints marked as $Far$;\\
\hspace*{15mm} set $V(\x) \; := \; q(\x)$ for all $\x \in X$;\\
\hspace*{15mm} mark all local minima of $q$ as $Considered$;

\noindent
\hspace*{15mm} while ($Considered$ list is not empty) \; \{\\
\hspace*{20mm} let $\xBar$ be such that $V(\xBar)$ is the smallest $Considered$ value;\\
\hspace*{20mm} mark $\xBar$ $Accepted$;\\
\hspace*{20mm} for each not-yet-$Accepted$ $\x_{i,j} \in N(\xBar)$ \; \{\\
\hspace*{25mm} update $V(\x_{i,j})$;\\
\hspace*{25mm} if $\x_{i,j}$ is $Far$, mark it $Considered$;\\
\hspace*{20mm} \}\\
\hspace*{15mm} \}
\end{algorithm}

A simple implementation of ``update $V(\x_{i,j})$'' is obtained by
re-solving \eqref{eq:DP_cUHP} using all current (possibly temporary) values
in $NV_{i,j}$.  A more efficient version, using the $Accepted$ subset of $NV_{i,j}$
and taking advantage of
the fact that only $V(\xBar)$ has recently changed, is described in
Remark \ref{rem:good_update} in Appendix B.

\section{Numerical examples}
\label{s:examples}

\subsection{Convergence study: a trivial free boundary.}
\label{ss:convergence1}
We start by studying convergence of our method on a simple
randomly-terminated continuous
example, where the free boundary is trivial and the analytic formula
for the solution is available.

Suppose $\cdomain = [-2,2]\times[-2,2]$, $q(\x) = |\x|$,
$K=0$, and $f=1$.  Note that $K=0$ and the fact that $q$ has only
one minimum imply that the free boundary is trivial; i.e.,
$\M = \B = \{ \bm{0} \}$, the global minimum of $q$.

Since $f$, $K$, and $q$ are radially symmetric, then so is the value function:
for every $\x \neq \bm{0}$ it is optimal to move along the straight line
toward the origin; i.e., the optimal control is $\ba^*(t) = -\x/|\x|$.
The expected cost of using this control (with the specified $q$ and $f$ and
with any radially symmetric $K$) is
\begin{equation}
\J(\x, \ba^*(\cdot))  \; = \; \int_0^{|\x|}
e^{-\lambda s} \left[
K \left(\x \,\frac{|\x| - s}{|\x|} \right)
+ \lambda \left( |\x| - s \right) \right] \, ds.
\label{eq:traj_to_center}
\end{equation}
In particular, when $K=0$,
$$
v(\x) = \J(\x, \ba^*(\cdot)) = \int_0^{|\x|}
\lambda e^{-\lambda s} \left( |\x| - s \right)  \, ds \; = \;
|\x| - \frac{1}{\lambda} \left(1 - e^{-\lambda |\x|} \right).
$$

Table \ref{t:err_trivial} lists the numerical errors observed in this computational example
for $\lambda = 0.5$.  The second column reports the maximum error
observed on the horizontal gridline passing through the origin.
(Since the characteristics are straight lines in this example,
this is equivalent to conducting the same experiment in 1D on
the domain $[-2,2]$.)  The remaining two columns report $L_2$
and $L_{\infty}$ errors computed on the entire $\cdomain$;
the $L_2$ errors are normalized to account for the non-unit area
of $\cdomain$.  The data clearly indicates the first order of convergence.

\begin{table}[hhhh]
\center{
\begin{tabular}{|l|c|c|c|} \hline
grid points & 1D $L_\infty$ error & 2D $L_2$ error & 2D $L_\infty$ error\\
\hline
101x101		& 0.0073	& 0.0062	& 0.0449\\
201x201		& 0.0037	& 0.0035	& 0.0259\\
401x401		& 0.0018	& 0.0020	& 0.0147\\
801x801		& 0.0009	& 0.0011	& 0.0083\\
1601x1601	& 0.0005	& 0.0006	& 0.0046\\
\hline
\end{tabular}
}
\caption{Errors for the trivial free boundary example with $\lambda = 0.5$.}
\label{t:err_trivial}
\end{table}

\begin{remark}
We note that for an Eikonal equation with a point source,
the rate of convergence is often found to be lower
due to a local non-smoothness of the viscosity solution at that point source.
The techniques to recover the first order of convergence include
pre-initializing the true solution in a disk of fixed radius centered at
the point source as well as the more recently proposed ``singularity removal''
method \cite{FomelLuoZhao}.
In the example considered above,
we do not face similar issues simply because the leading
term in the expansion of $v$ is $|\x|^2$ rather than $|\x|$.
\end{remark}

\subsection{Convergence study: a circular free boundary.}
\label{ss:convergence2}

A small modification of the previous example already leads to
a non-trivial free boundary.  We let $K(\x) = |\x|$ and take
the same $f$, $q$ and $\cdomain$ as above.

As before, $v(\x)$ is radially symmetric, and if $\x \not \in \M$,
then the optimal trajectory starting from $\x$ is again $\ba^*(t)$,
whose expected cost is evaluated from the integral in \eqref{eq:traj_to_center}:
$$
\J(\x, \ba^*(\cdot))  \; = \;
\frac{\lambda + 1}{\lambda}
\left( |\x| - \frac{1}{\lambda} \left(1 - e^{-\lambda |\x|} \right) \right).
$$
The value function is
$v(\x)  = \min \left( q(\x), \J(\x, \ba^*(\cdot)) \right)$, and, unlike in the previous example,
$q(\x) = |\x|$ is actually smaller when $|\x|$ is sufficiently large.
Thus, for this example, $\B = \{ \bm{0} \} \cup \C(r)$, where
$\C(r) = \{ \x  \mid |\x| = r \}$ and the radius $r$ is such that
\begin{equation}
\frac{\lambda + 1}{\lambda}
\left( r - \frac{1}{\lambda} \left(1 - e^{-\lambda r} \right) \right) =  r.
\label{eq:FBradius}
\end{equation}
The origin is the inflow part and $\C(r)$ is the outflow part of $\B$.

This example also illustrates the asymptotic behavior of $\B$.
Applying definitions of asymptotic motionless sets from section \ref{ss:contin_value_func},
$$
\M^0 = \left\{(0,0)\right\} \bigcup \left\{ \x \, \mid \; q(\x)  \leq
\int_0^{|\x|} K \left(\x \,\frac{|\x| - s}{|\x|} \right) \, ds \right\}
\qquad \text{and} \qquad
\M^{\infty} = \left\{(0,0)\right\} \bigcup \left\{ \x \, \mid \; K(\x) - |\nabla q(\x)| \geq 0 \right\}.
$$

When $\lambda \to \infty$, Theorem \ref{thm:asymptotics_v} shows
that $\B_{out} \to \B_{\infty} = \{ \x  \mid |\x| = K(\x) = |\nabla q(\x)| = 1 \} = \C(1)$.
On the other hand, when $\lambda \to 0$, we have
$\B_{out} \to \B_0 = \{ \x  \mid |\x| = q(\x) = \int_0^{|\x|} (|\x| -s) \, ds  \} = \C(2)$.
This is also confirmed in Figure \ref{fig:FBradius}
obtained by solving \eqref{eq:FBradius} numerically.
Table \ref{t:err_trivial} lists the numerical errors observed in this example
for two different values of $\lambda$.
\begin{figure}[hhhh]
\center{
\includegraphics[width=15cm]{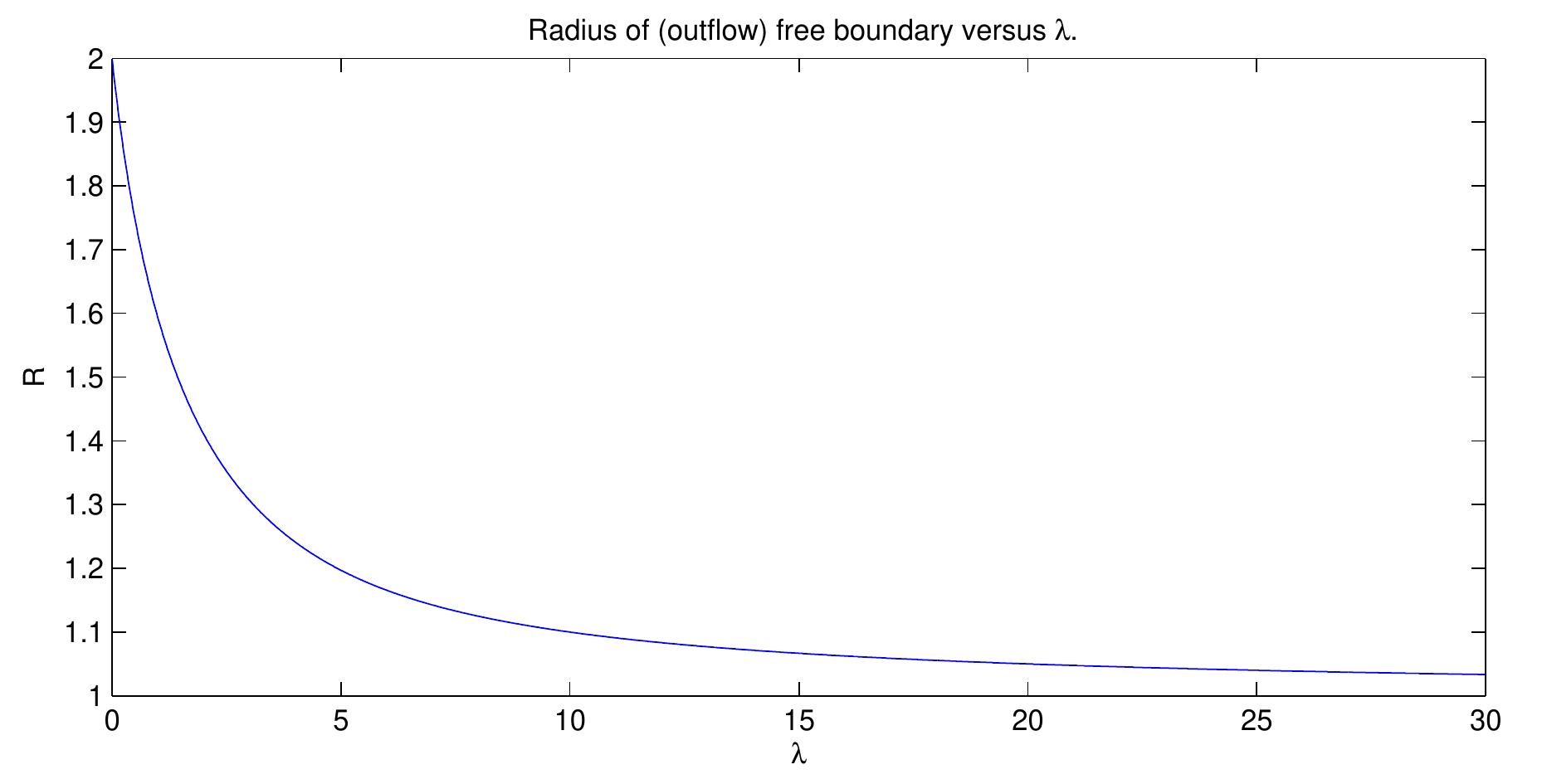}
}
\caption{Radius of the outflow free boundary for different values of $\lambda$.}
\label{fig:FBradius}
\end{figure}

\begin{table}[hhhh]
\center{
\begin{tabular}{|l||c|c|c||c|c|c|} \hline
& \multicolumn{3}{|c||}{$\lambda = 0.5$} & \multicolumn{3}{|c|}{$\lambda = 25$}\\
\hline
grid points & 1D $L_\infty$ error & 2D $L_2$ error & 2D $L_\infty$ error & 1D $L_\infty$ error & 2D $L_2$ error & 2D $L_\infty$ error\\
\hline
101x101		& 0.0216	& 0.0048	& 0.0344	& 0.0055	& 0.00015	& 0.0092\\
201x201		& 0.0109	& 0.0025	& 0.0173	& 0.0032	& 0.00008	& 0.0053\\
401x401		& 0.0055	& 0.0012	& 0.0087	& 0.0017	& 0.00004	& 0.0029\\
801x801		& 0.0027	& 0.0006	& 0.0044	& 0.0009	& 0.00002	& 0.0015\\
1601x1601	& 0.0014	& 0.0003	& 0.0022	& 0.0005	& 0.00001	& 0.0008\\
\hline
\end{tabular}
}
\caption{Errors for the circular (outflow) free boundary example.}
\label{t:err_circular}
\end{table}

\subsection{Optimal idle-time motion in continuous space}
\label{ss:contin_idle}
We now consider a continuous equivalent of the optimal idle-time
processing example of section \ref{ss:optimal_idler}.
An all-terrain vehicle moving in $\cdomain \subset R^2$
is tasked with responding to emergency calls.
We assume that the arrival of calls is a Poisson process
with rate $\lambda$,
and our goal is to minimize the expected response time
to the first caller.
We are given a list of possible caller locations
$\xBtilde_1, \ldots, \xBtilde_r$ inside $\domain$
and the corresponding probabilities
$\tilde{P}_1, \ldots, \tilde{P}_r$ of the next call originating
at each of these locations.  The vehicle's dynamics is assumed to
be isotropic; i.e.,
$\y' = f(\y) \ba$, where $\ba \in S_1$ is the current direction of motion.
We then use the Fast Marching Method to solve $r$ Eikonal problems:
$$
| \nabla u_i (\x) | f(\x) = 1, \quad \x \in \domain \backslash \{ \xBtilde_i \};
\qquad
u_i(\xBtilde_i) = 0, \quad \text{ and $u = +\infty$ on } \boundary;
\qquad i=1, \ldots,r.
$$
The resulting $u_i(\x)$ is the minimum time to reach $\xBtilde_i$
for a vehicle starting from $\x$ and constrained to move within $\domain$.
If the call is received at $\x$, the expected time from there
to the caller is
\begin{eqnarray}
q(\x) &=& \sum_{i=1}^r \tilde{P}_i u_i(\x).
\label{eq:q_response}
\end{eqnarray}
(We note that thus defined $q(\x)$ is always Lipschitz-continuous.)\\
A global minimum of $q$ is obviously the optimal place to ``park'' the vehicle
while expecting a call.  But what if the current position is not a global
minimum of $q$?  The most intuitive approach, based on a gradient descent in $q$,
is far from optimal.  (E.g., it would prescribe not moving away from any
local minimum of $q$ -- a clearly bad strategy when $\lambda$
is small enough and there is a good chance of reaching the global minimum of $q$
before the next call.)
Since we are trying to minimize the expected response time,
it is logical to set $K=0$, since in this continuous-time control,
the vehicle starts responding instantaneously.  (This is in contrast to
the discrete-transitions scenario considered in section \ref{ss:optimal_idler}.)

We consider an example in which $f$ is piecewise-constant: $f=0.2$
in a large circular ``slow'' region in the center and $f=1$ everywhere else.
We use 4 different call locations ($\xBtilde_1, \ldots, \xBtilde_4$)
symmetric relative to the slow region and numbered
counter-clockwise starting from the lower left;
see Figures \ref{fig:circular_slow_lambda_small} and \ref{fig:circular_slow_lambda_25}.

We perform this experiment with several sets of parameter values.
First, we set $\lambda=0.05$, $\tilde{P}_1 = \tilde{P}_2 = \tilde{P}_3 = 0.2$
and $\tilde{P}_4=0.4$.  The results in Figure \ref{fig:circular_slow_lambda_small}
highlight the differences between $q$ and $v$.

\begin{figure}[H]
\center{
$
\hspace*{-1cm}
\begin{array}{cc}
\includegraphics[width=9cm]{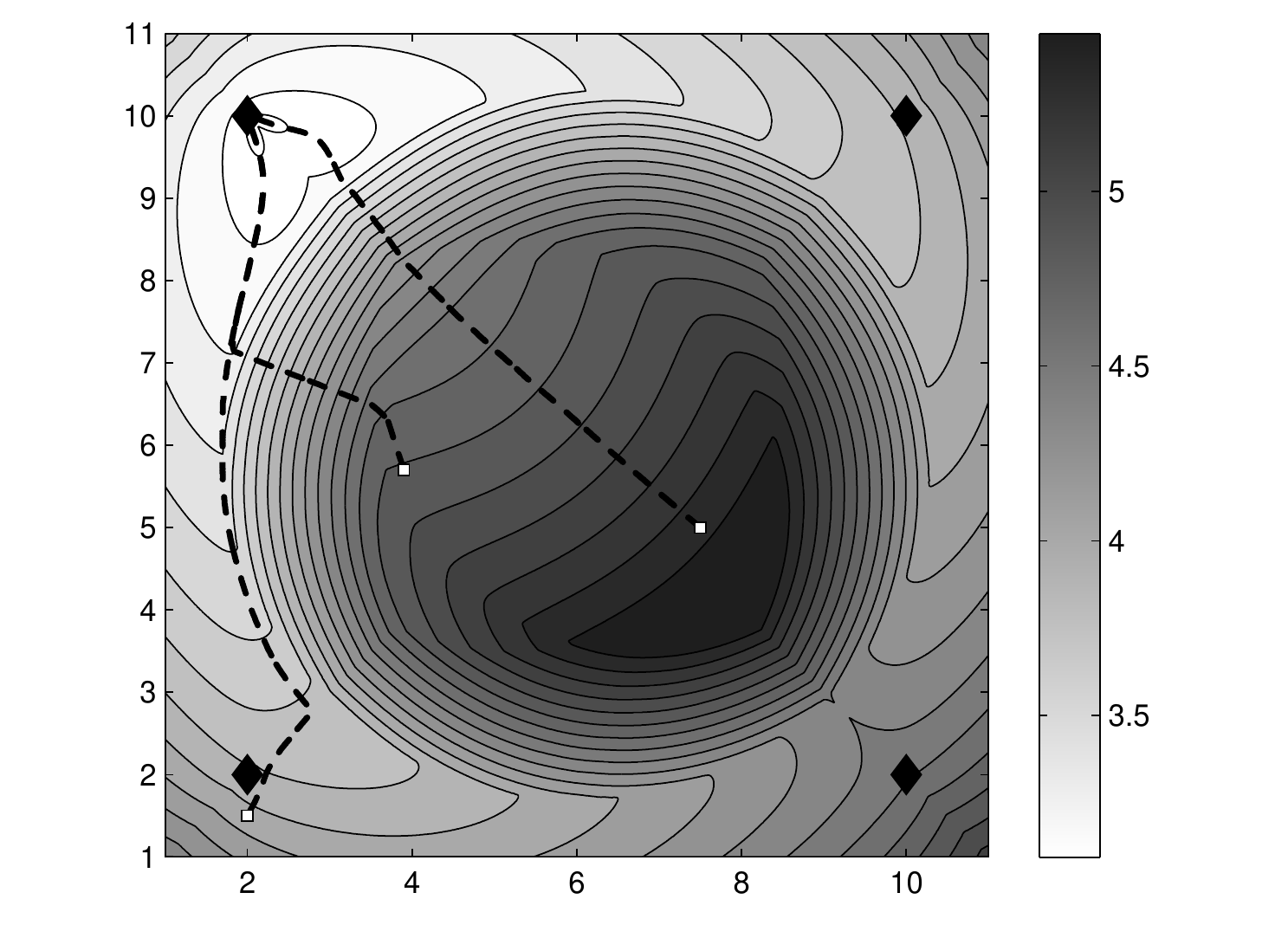}&
\includegraphics[width=9cm]{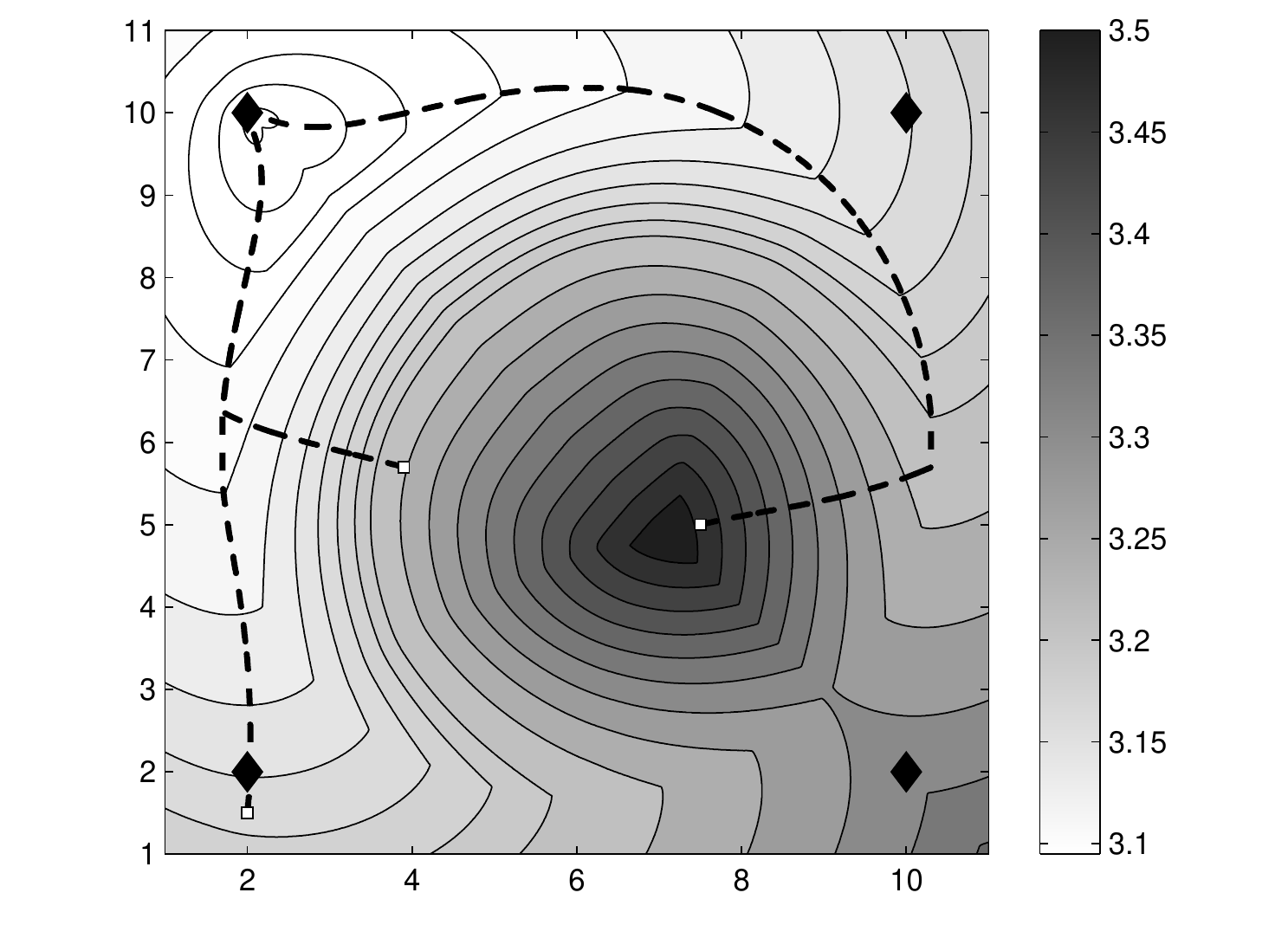}\\
A & B
\end{array}
$
}
\caption{
{\footnotesize
$\lambda=0.05 \,$
Four emergency call locations (shown by black diamonds) around a circular
``slow region''.
Three sample starting locations are shown by small white squares.
Level curves of $q$ (left) and of $v$ (right)
shown by solid lines.  ``Optimal'' trajectories
(shown by dotted lines)
found by gradient descent in $q$ (on the left)
are quite different from the truly optimal
trajectories found by gradient descent in $v$ (on the right).
}
}
\label{fig:circular_slow_lambda_small}
\end{figure}

We then set $\lambda=25$ and repeat the experiment two more times:
with equal probabilities ($\tilde{P}_1 = \ldots = \tilde{P}_4 = 0.25$)
in Figure \ref{fig:circular_slow_lambda_25}A and with varying probabilities
($\tilde{P}_1 = \tilde{P}_2 = 0.2; \; \tilde{P}_3 = 0.25; \; \tilde{P}_4 = 0.35$)
in Figure \ref{fig:circular_slow_lambda_25}B.
In all three cases the computations are performed on a $1001 \times 1001$  grid
and the level sets are selected to highlight the complex structure of
the solution outside of the slow region.
As the figures show, two nearby starting locations can easily produce
dramatically different optimal trajectories when the locations are on
different sides of the shock line (where $\nabla v$ is undefined).

\begin{figure}[hhhh]
\center{
$
\hspace*{-1cm}
\begin{array}{cc}
\includegraphics[width=9cm]{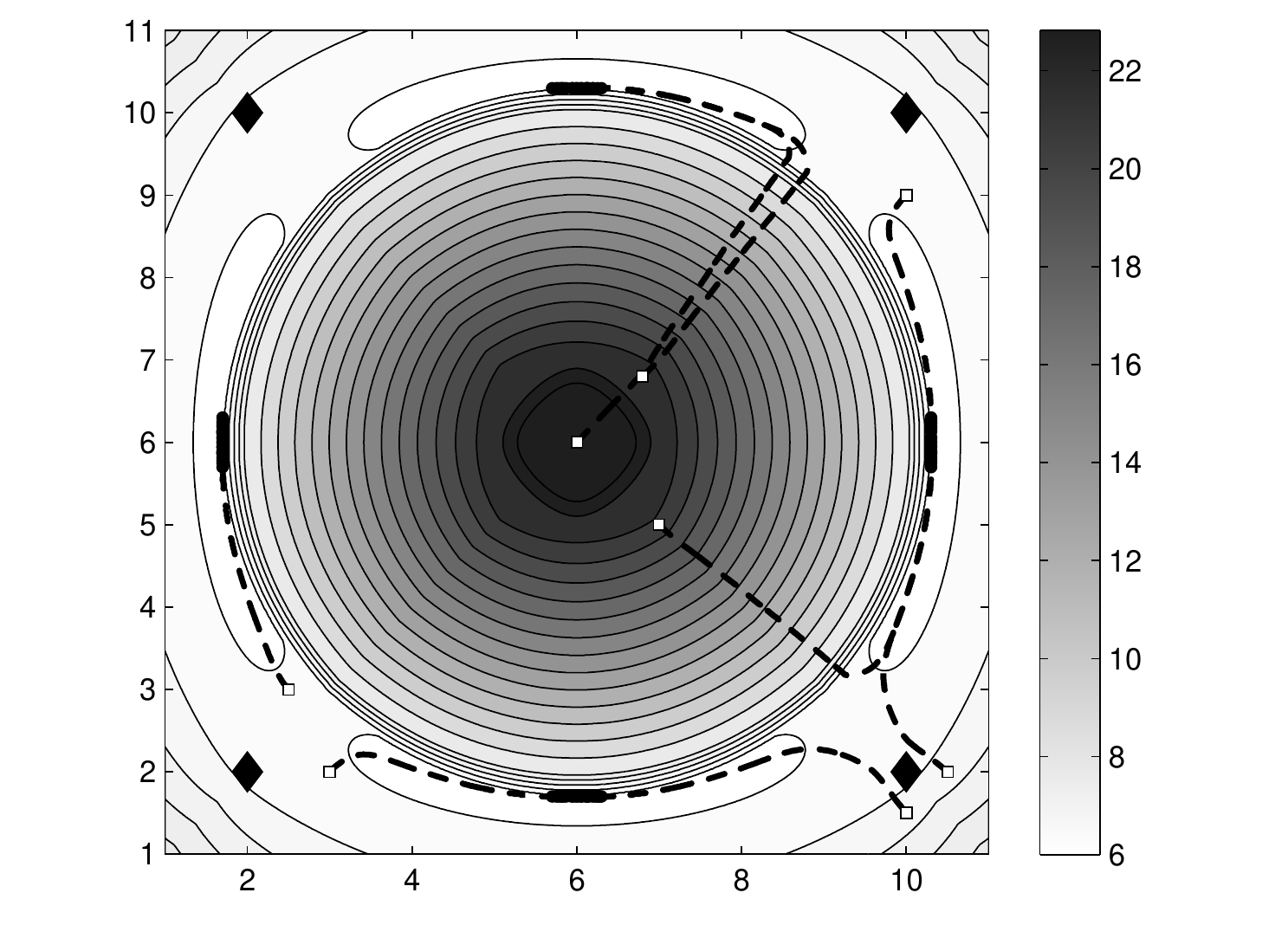}&
\includegraphics[width=9cm]{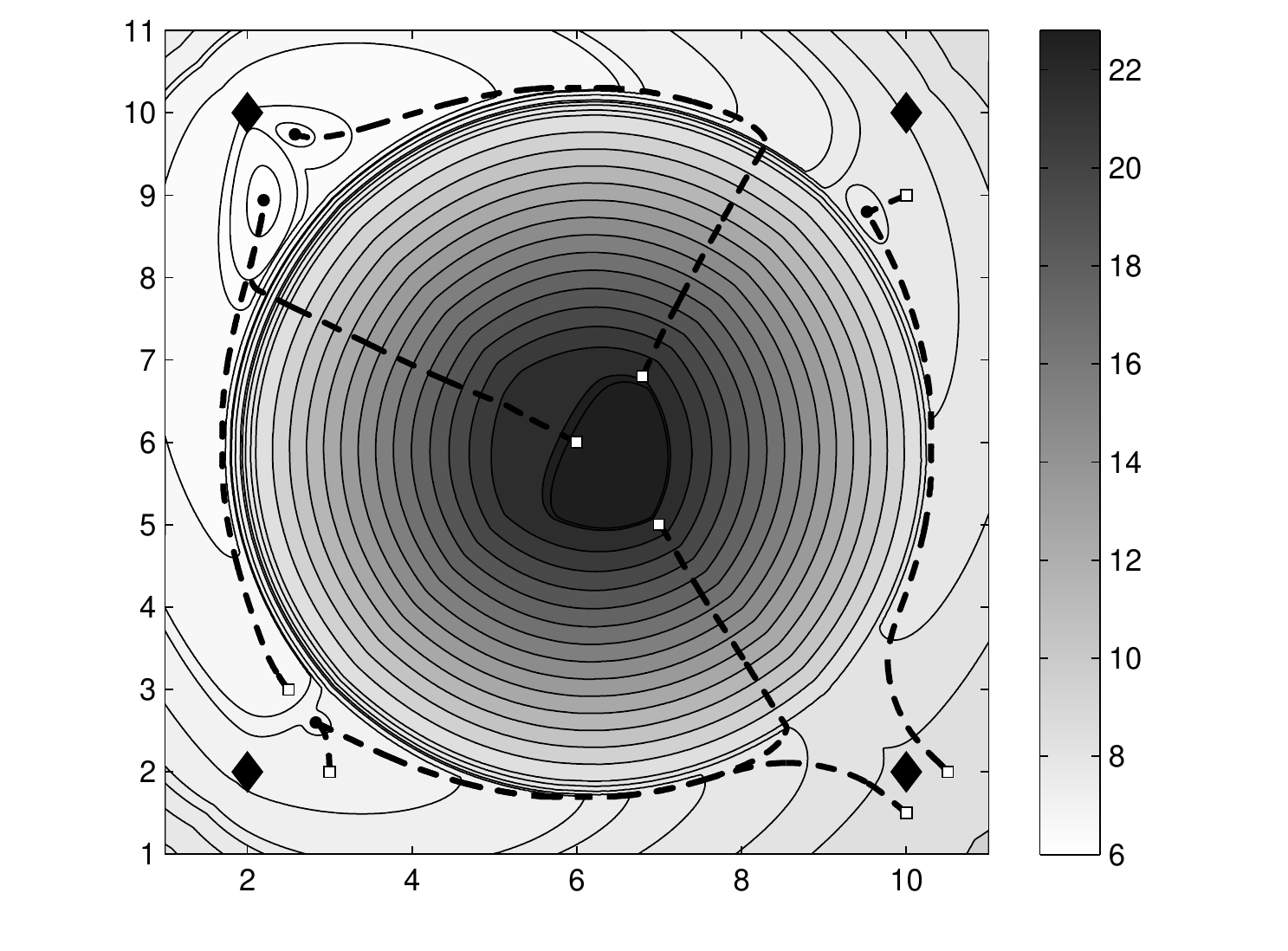}\\
A & B
\end{array}
$
}
\caption{
{\footnotesize
$\lambda=25. \,$
Four emergency call locations (shown by black diamonds) around a circular
``slow region''.  Level curves of $v$ shown by solid lines.
Optimal trajectories are shown by dotted lines, running from several starting locations
(shown by small squares) to the free boundary $\B$ (found numerically and
shown by thick black dots).
}
}
\label{fig:circular_slow_lambda_25}
\end{figure}

\begin{remark}
We note that $K \equiv 0$ leads to $\M$ consisting of (a subset of) local minima
locations of $q$; these are found numerically (and indicated by thick black
dots in the corresponding figures).
As a result, the true shape of $\M$ may not become apparent
even on fairly refined computational grids.
For example, when all caller locations are equally likely, it
is relatively easy to show analytically that the minima of $q$ are
attained at four isolated points only, while
Figure \ref{fig:circular_slow_lambda_25}A shows a larger motionless set.
This, however, is due to the fact that the value function varies very
slowly in this region (e.g., order of $10^{-6}$ variation of $q$
on the numerically found $\M$).  This also presents an additional challenge
in recovering optimal trajectories, which for isotropic problems
is done by a gradient descent in $v$.  The characteristic
equations show that, when $K \equiv 0$, the directional derivative of $v$ along the
optimal trajectory becomes zero at $\B$.
To circumvent this difficulty, our current implementation forces
the trajectory to take a straight line
path to $\B$, when the distance to it decreases below an
$h$-dependent threshold.
No such heuristic adjustments are needed when $\widebar{K}_1 >0$.
\end{remark}

\subsection{Navigating a maze}
\label{ss:maze}

The following example illustrates the effect of changing $\lambda$
on the free boundary for problems with non-zero running cost $K$.

We assume two possible locations for emergency calls:
$\xBtilde_1 = (1.0, 0.1)$ and $\xBtilde_2 = (9.0, 0.1)$
with the corresponding probabilities of calls
$\tilde{P}_1 = 0.2$ and $\tilde{P}_2 = 0.8$	
inside the domain $\cdomain = [0,10] \times [0,10]$.
We further assume that the domain contains a ``maze''
with $K$ high / $f$ low within its ``walls''
and $K$ low / $f$ high everywhere else.
Figure \ref{fig:Maze_setup} shows the described $f$, $K$, and
the resulting terminal cost function $q$ computed from
\eqref{eq:q_response}.

\begin{figure}
\begin{center}
$
\hspace*{-1cm}
\begin{array}{ccc}
\includegraphics[width=6cm]{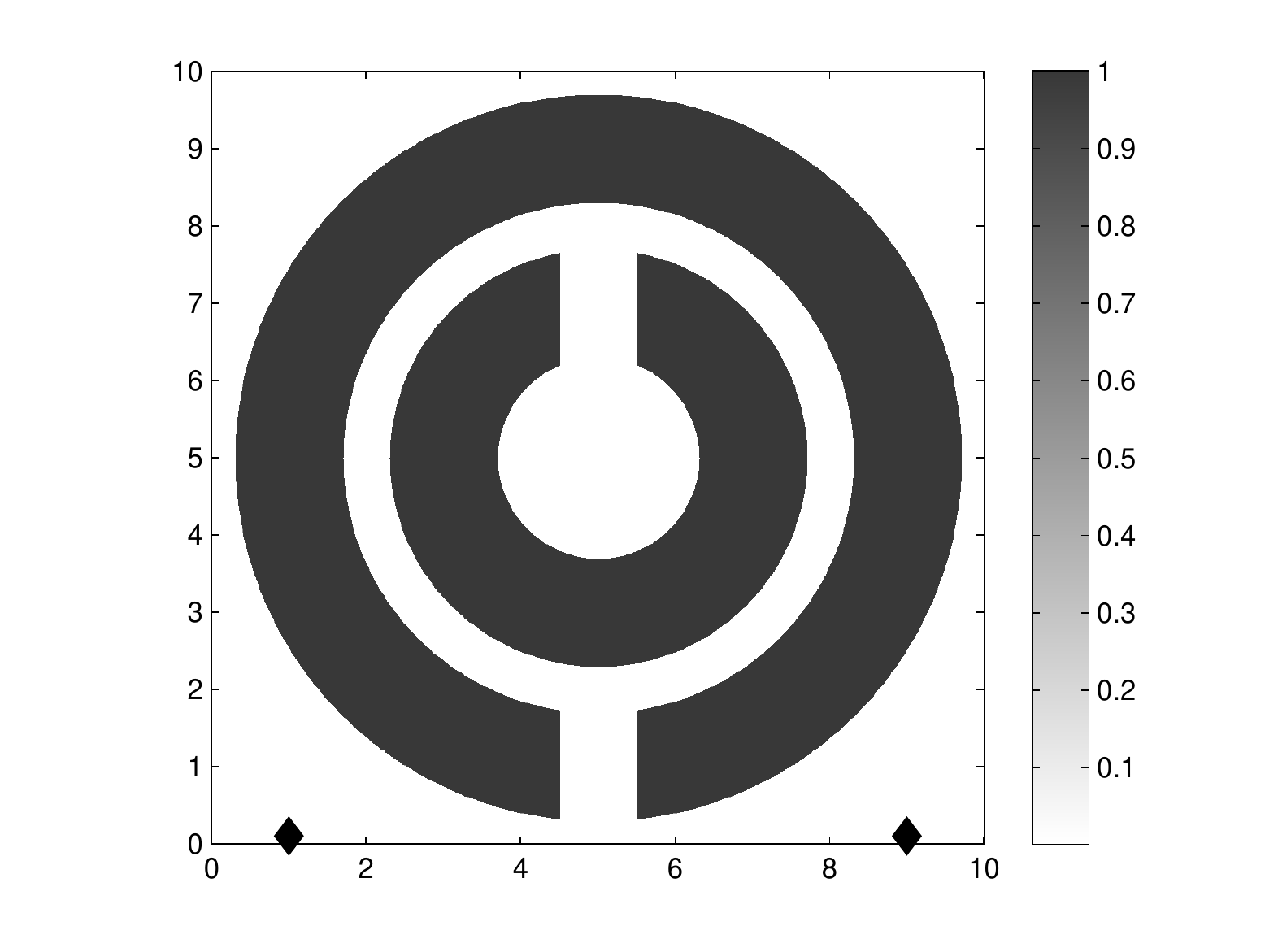}&
\includegraphics[width=6cm]{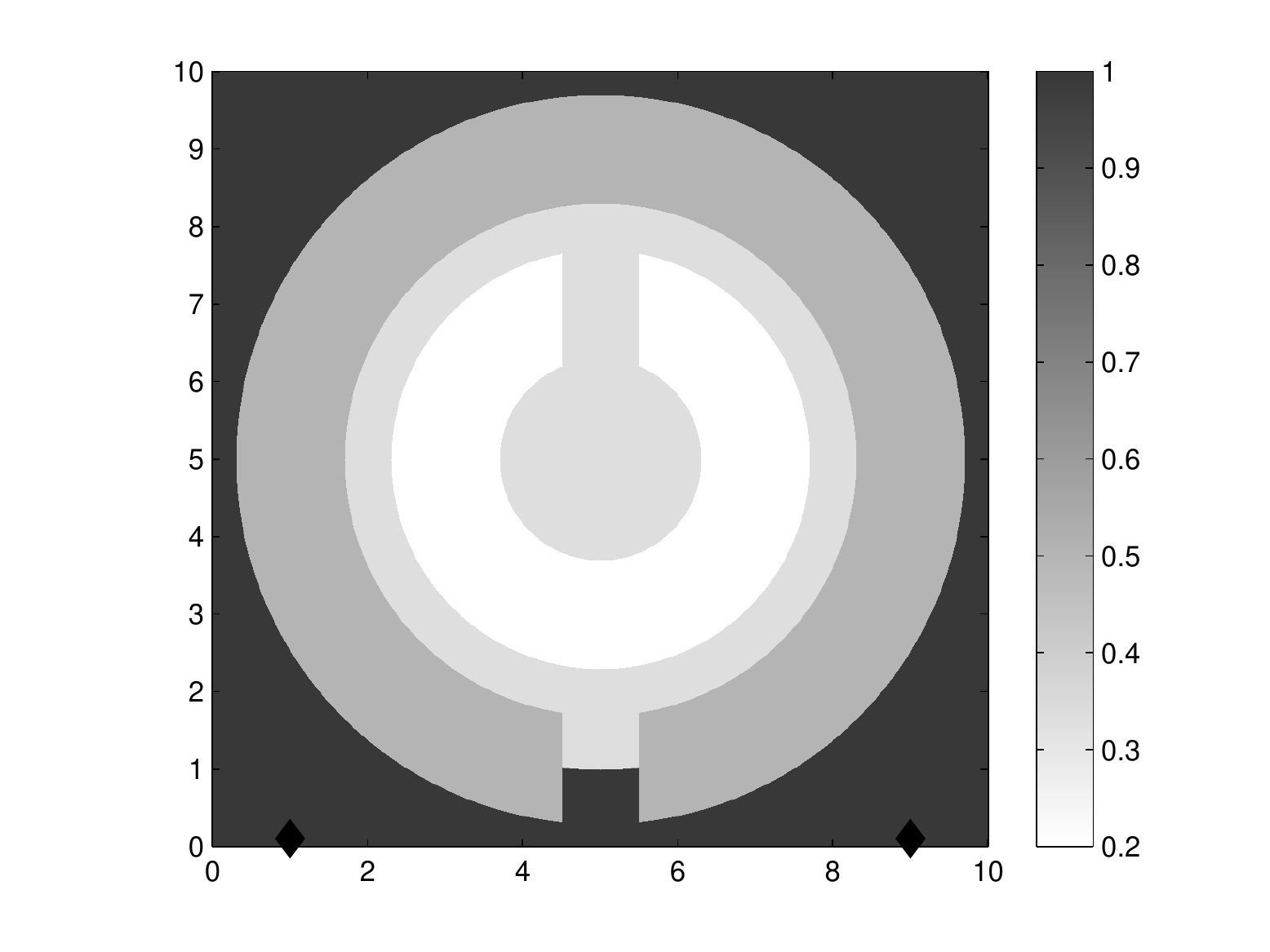}&
\includegraphics[width=6cm]{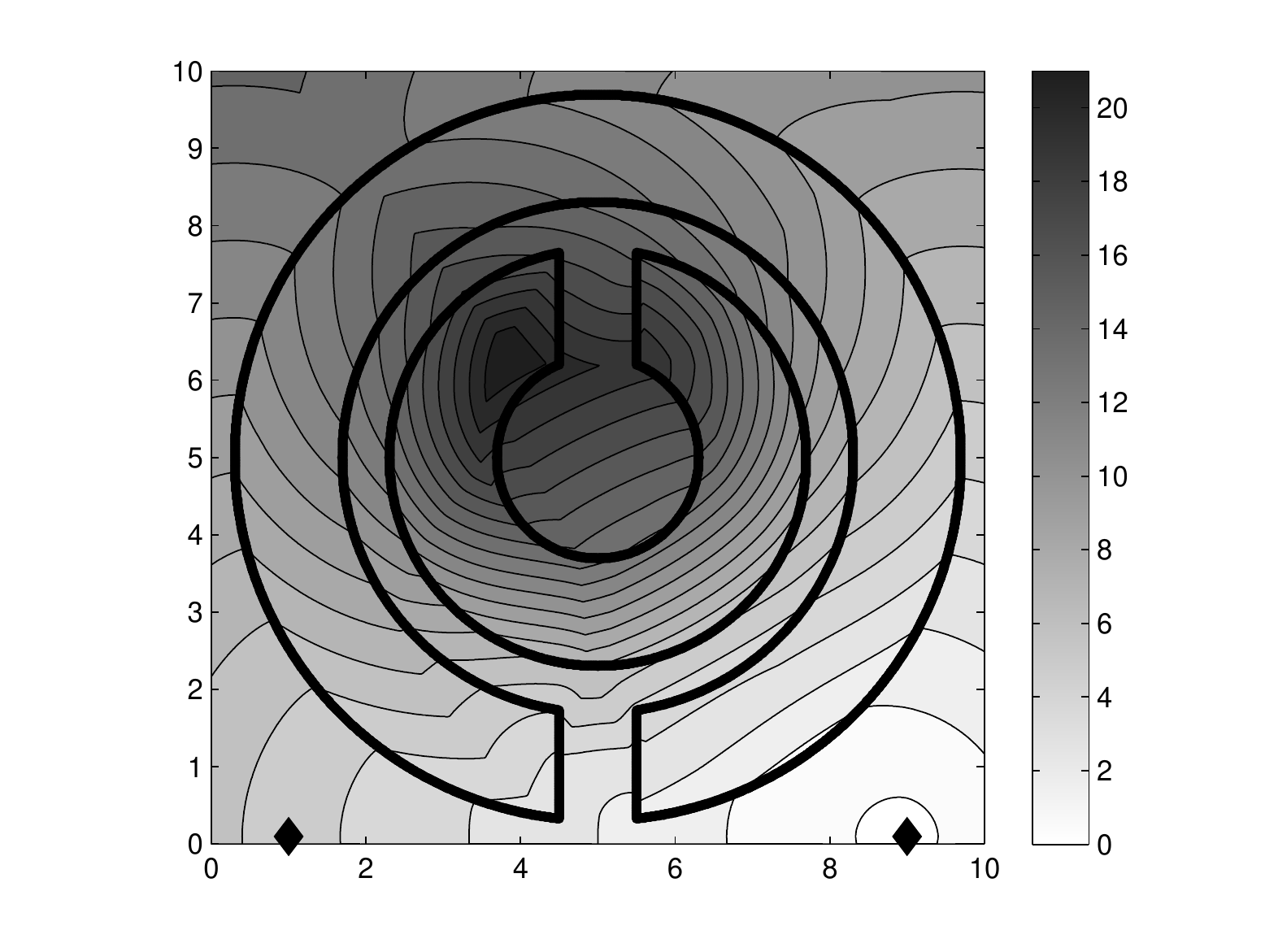}
\end{array}
$
\end{center}
\caption{ The ``maze example'' : running cost $K$ (Left), speed $f$ (Center), and
the terminal cost $q$ (Right).  The caller locations are shown by black diamonds.
The thick solid lines in the last subfigure show  $\partial \M^{\infty}$.}
\label{fig:Maze_setup}
\end{figure}

Figure \ref{fig:Maze_solutions} shows the level curves of $v$ and
the optimal trajectory starting from the center of the maze
computed for different values of $\lambda$.
The free boundary $\B$ is also indicated by a thick line in each case.
We note that the running cost matters only until
the emergency call arrives.  Thus,
for very small $\lambda$'s, it is likely
that the global minimum of $q$ (i.e., the point
$\xBtilde_2$) can be reached avoiding the walls before the first call arrives.
When $\lambda$ increases, it becomes more optimal to head
toward (and then through)
the closest wall, hoping that the call arrives before we reach it.
Finally, for large $\lambda$'s
it is optimal to stop and wait for the call at the wall boundary.

\begin{figure}
\begin{center}
$
\begin{array}{cc}
\includegraphics[width=9cm]{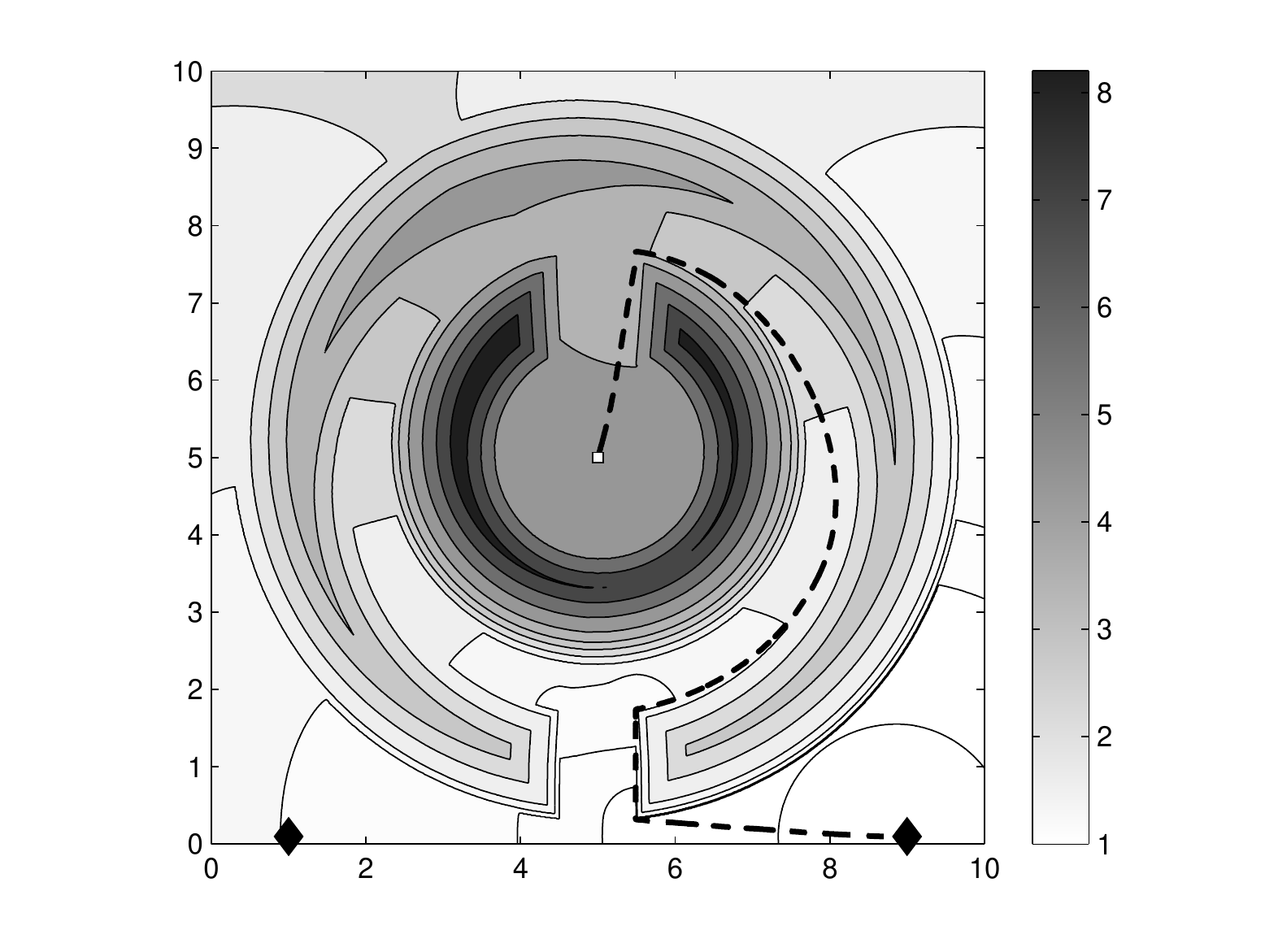}&
\includegraphics[width=9cm]{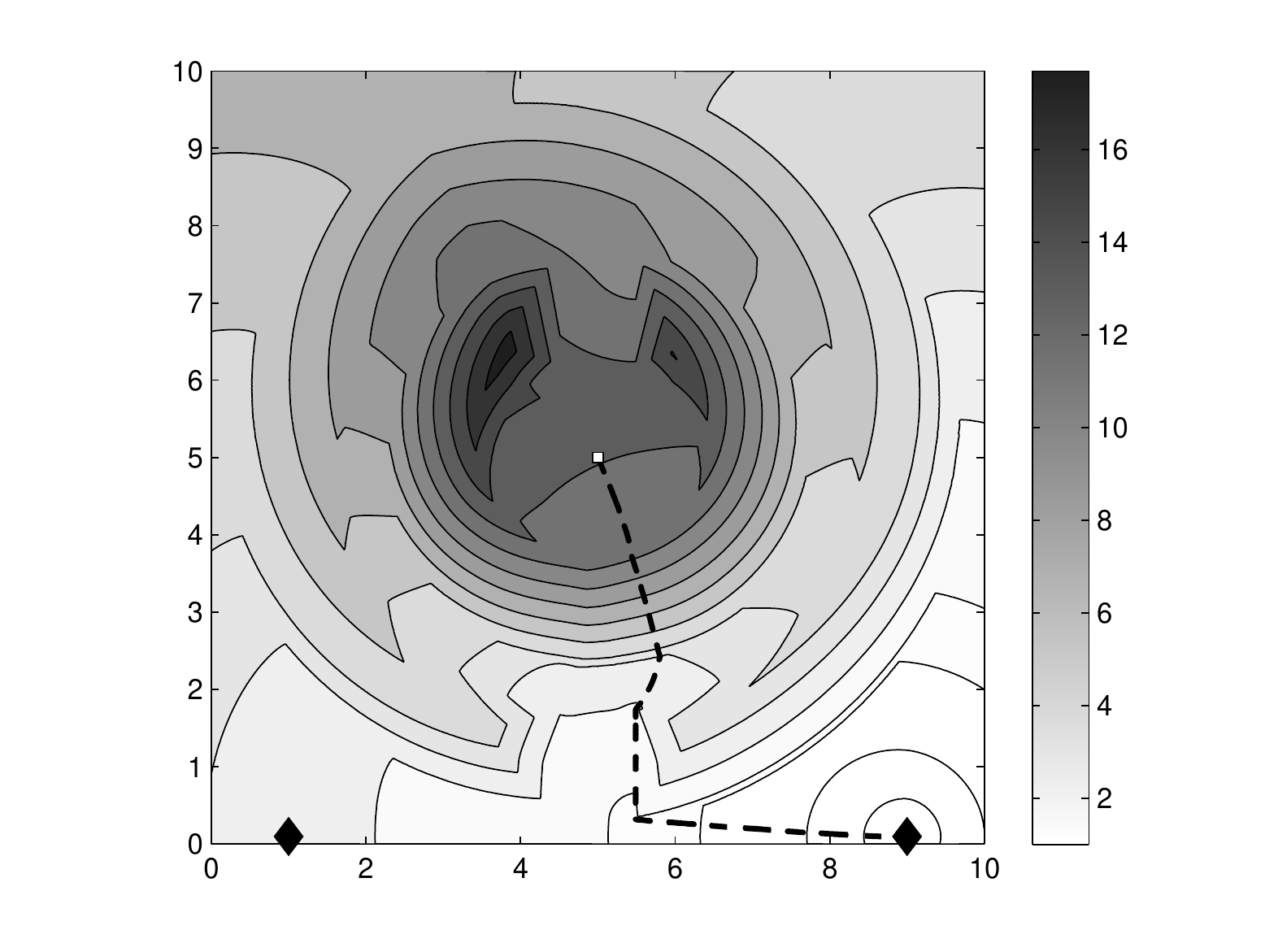}\\
\lambda = 0.01 &
\lambda = 0.1\\
\\
\includegraphics[width=9cm]{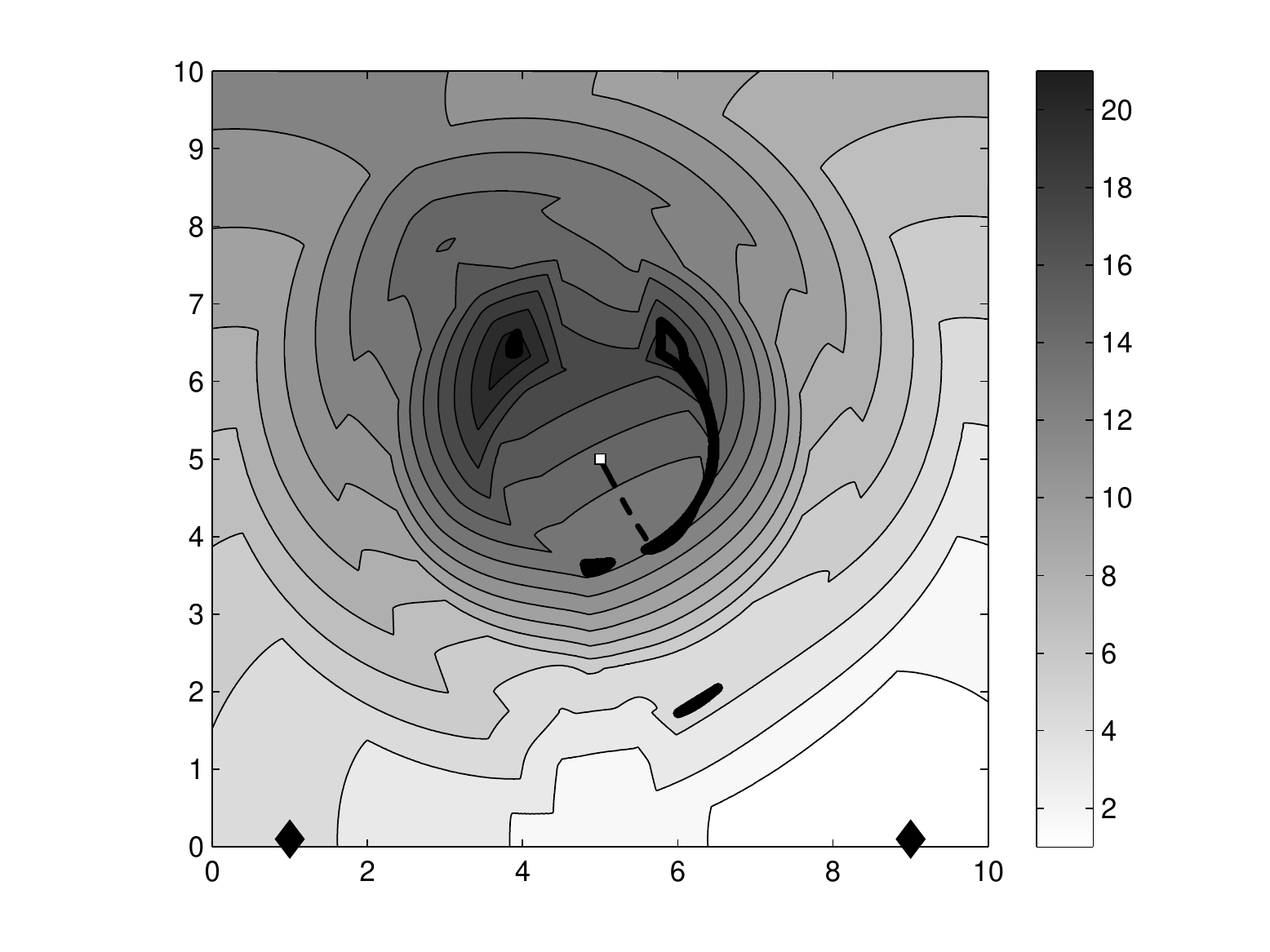}&
\includegraphics[width=9cm]{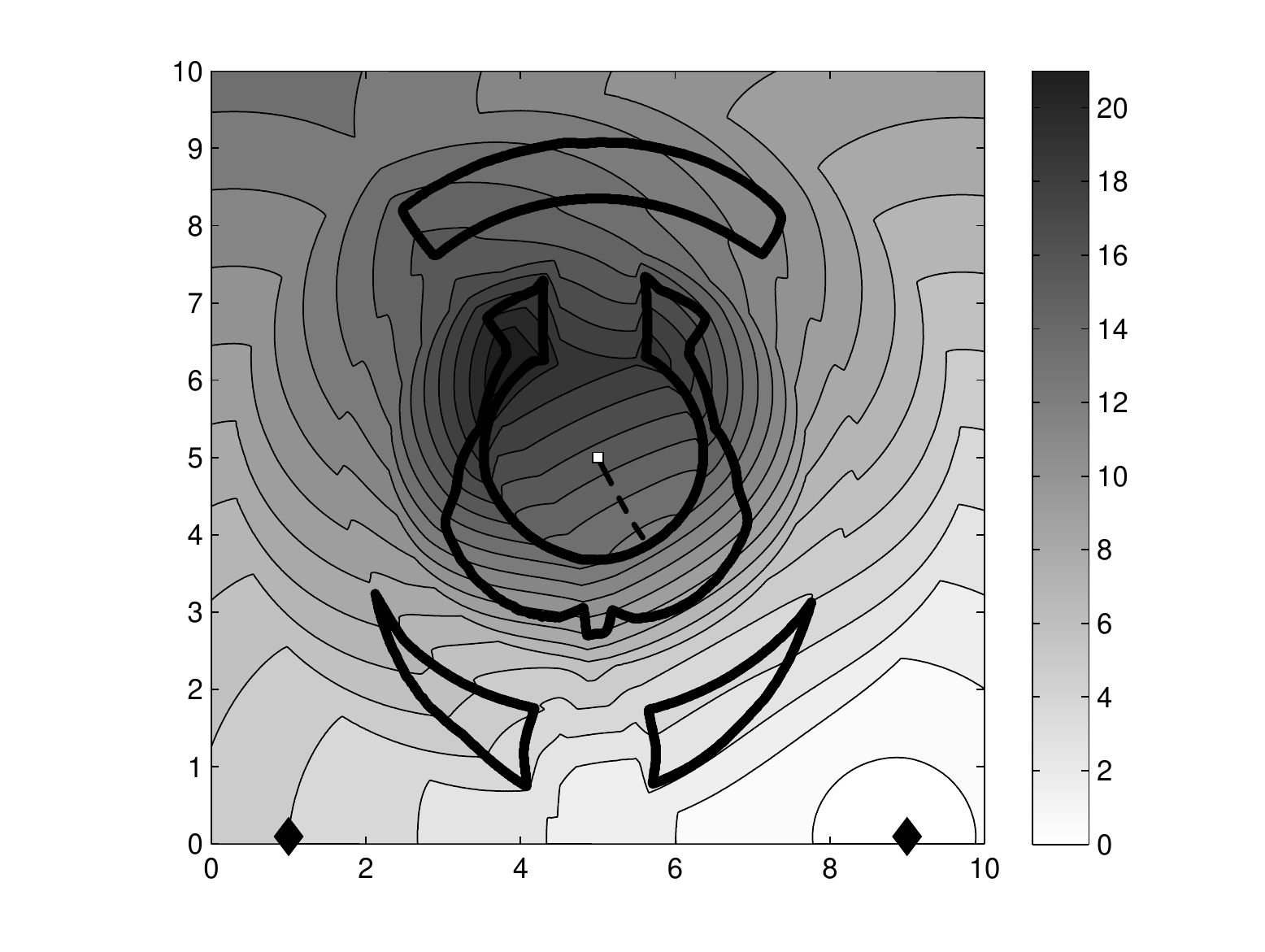}\\
\lambda = 0.45 &
\lambda = 1.5
\end{array}
$
\end{center}
\caption{ The ``maze example'' : level-curves of the value function
$v(\x)$ computed for 4 different values of $\lambda$.
In each case, the optimal trajectory starting from the center is
indicated by a dotted line.
For larger $\lambda$'s, when $\M$ is not just the global minimum of $q$,
the free boundary is shown by a thick solid line.}
\label{fig:Maze_solutions}
\end{figure}

\section{Conclusions.}
\label{s:conclusions}

We have considered a wide class of uncertain-horizon problems
and showed that non-iterative methods can be used to compute
their value functions both in discrete and continuous settings.
The numerical examples in section \ref{s:examples} have illustrated
both the convergence properties and the asymptotic behavior of
the free boundary $\partial \M$.
Our modification of the Fast Marching Method in section \ref{ss:FMM_modified}
addressed the isotropic case only, but similarly modified
Ordered Upwind Methods \cite{SethVlad2, SethVlad3, AltonMitchell2}
can be used to treat the
anisotropic cost and dynamics in randomly-terminated problems.
A Dial-like version of the Fast Marching Method
will be similarly applicable if \eqref{eq:Eikonal_var_ineq} is discretized
on an acute triangulated mesh \cite{VladMSSP}.
Another fairly straightforward generalization
is to treat inhomogeneous termination rates
in the continuous case; i.e., $\lambda = \lambda(\x)$;
this would require only minimal changes to the label-setting algorithms.

In this paper we have not discussed the label-correcting methods
\cite{Bertsekas_NObook}, whose asymptotic complexity is worse but practical
performance is sometimes better than that of label-setting methods.
Their applicability to uncertain-horizon problems is clearly
also of interest.  In the continuous case, we believe that
fast sweeping methods (e.g., \cite{BoueDupuis, TsaiChengOsherZhao, Zhao})
and various fast-iterative methods
(e.g., \cite{PolyBerTsi, BorRasch, Renzi, JeongWhitaker})
should be easy to extend to randomly-terminated problems.
The same is also true for hybrid two-scale methods that aim
to combine the advantages of marching and sweeping \cite{ChacVlad, ChacVlad2}.
Careful testing would be needed to compare the computational
efficiency of these alternatives to that of Dijkstra-like methods considered here.

The second author has previously studied the so-called Multimode SSP (MSSP) in \cite{VladMSSP}
and derived 
sufficient conditions for the applicability of
the label-setting methods to them.  We emphasize that
the uncertain-horizon problems considered here cannot be cast as MSSPs
and the results from \cite{VladMSSP} do not apply.
However, the results from both papers can be easily combined
to address the randomly-terminated MSSPs.
More generally, we believe that label-setting methods will
be applicable for a broader class of {\em hierarchically causal} SSPs.
A natural example from this category is the ``SSP with recourse''
problem, where the transition-costs are known
only probabilistically at first
and their true values are learned in the process of
traveling through the graph.
A Dijkstra-like method for such problems was previously introduced
by Polychronopoulos and Tsitsiklis in \cite{TsiPoly}.

On the continuous side, the randomly-terminated problems form a
simple subclass of {\em piecewise-deterministic} problems \cite{Davis, Haurie}.
The latter arise when a Poisson process governs random switches
between several known types of deterministic dynamics/cost,
yielding a system of weakly-coupled static non-linear HJB PDEs.
We believe that these more general problems can also be treated
by non-iterative numerical methods provided all
non-deterministic transitions (and the resulting couplings between the PDEs)
are hierarchically causal.

Finally, we note that in all randomly-terminated problems considered
in the current paper,
the only goal was to minimize the expected value of the total cost.
An interesting direction for future work is to incorporate
simultaneous optimization and/or constraints based on the worst case scenario.
Efficient algorithms for this more difficult problem can be built
using the recent method for
multiobjective optimal control introduced in \cite{KumarVlad}.

\vspace{.2in}
\noindent
{\bf{Acknowledgments:}}
The authors are grateful to anonymous reviewers for their suggestions on improving this paper.
The second author would also like to thank A. Dukhovny and V. Protsak.

\section{Appendix A: optimality of stationary policies.}
\label{s:Appendix_A}

This section contains the results on existence of optimal stationary policies
for the general randomly-terminated processes on graphs.
We note that the following proofs do not make any use of assumptions
\Assume{1}-\Assume{3}; i.e., both the transition penalties $K_{ij}$ and
the terminal costs $q_j$ can be positive or negative and the self-transitions
need not be allowed.

A function $\mu : X \mapsto X$ is a {\em control mapping}
if $\mu(x) \in N(x)$ for $\forall x \in X$.
A {\em policy} is an infinite sequence of control mappings
$\pi = (\mu_0, \mu_1, \ldots)$.
Starting from any $x \in X$
a policy will generate a particular path $\y^{\pi}$ as follows: $y^{\pi}_0 = x$,
$y^{\pi}_{k+1}  = \mu_k(y^{\pi}_k)$ for $\forall k \geq 0$.
Defining $\J(x, \pi) = J(\y^{\pi})$, we can also re-write
the value function as
\begin{equation}
\label{eq:value_through_policies}
V(x) = \min\limits_{\pi} \J(x, \pi).
\end{equation}
If $\mu = \mu_0 = \mu_1 = \ldots$, the corresponding policy $\pi = (\mu, \mu, \ldots)$
is called {\em stationary}.  (We will also somewhat abuse the notation and refer to a
stationary policy $\mu$.)  Since stationary policies generate only simple paths,
Theorem \ref{thm:stationary} proves the existence of an optimal stationary policy.
Formula \eqref{eq:path_cost_recursive} shows that for any stationary policy,
$\J(x, \mu) = K(x, \mu(x)) + p q(\mu(x)) + (1-p) \J(\mu(x), \mu))$,
which implies the dynamic programming equation \eqref{eq:DP_uncertain}.

\begin{remark}
\label{rem:relation_to_SSP}
Our problem can be easily recast as a Stochastic Shortest Path (SSP) problem
by adding a special absorbing terminal node $x_{\bm{t}} = x_{M+1}$,
and considering $N(x_i)$ to be the set of controls available at $x_i$.
A choice of the control $x_j \in N(x_i)$ then results in a transition
to the node $x_j$ with probability $(1-p)$ and to
the node $x_{\bm{t}}$ with probability $p$.  The cost associated with
this control is $(K_{ij} + p q_j)$.
The process terminates upon reaching $x_{\bm{t}}$.
In \cite{BertsTsi} Tsitsiklis and Bertsekas proved the equivalents of Lemma \ref{lm:min_attained} and
Theorem \ref{thm:stationary} for a broader class of general SSPs.
However, in our setting, direct proofs are much simpler and
exploit the special structure of this problem.
\end{remark}

We first note the following useful generalization of
the recursive formula (\ref{eq:path_cost_recursive}).
Starting from (\ref{eq:infinite_horizon}),
\begin{eqnarray}
\nonumber
J(\y) = E[Cost(y_0, \ldots, y_k, \ldots)]
&=& \sum_{i=0}^{k-1} \left( K(y_i, y_{i+1}) + p q(y_{i+1}) \right) (1-p)^i
\; + \;
\sum_{i=k}^{\infty} \left( K(y_i, y_{i+1}) + p q(y_{i+1}) \right) (1-p)^i\\
&=& \sum_{i=0}^{k-1} \left( K(y_i, y_{i+1}) + p q(y_{i+1}) \right) (1-p)^i
\; + \;
(1-p)^k \, J(y_k, y_{k+1}, \ldots).
\label{eq:path_recursive_general}
\end{eqnarray}

\begin{lemma}
For every $x \in X$, there exists $\ybar \in Y(x)$ such that
$J(\ybar) = \inf\limits_{\y \in Y(x)} J(\y).$
\label{lm:min_attained}
\end{lemma}
\begin{proof}
Given any two paths $\y, \tilde{\y} \in Y$,
let $k$ be the first stage of the process where the paths become different;
i.e., $k(\y, \tilde{\y}) = \min \{j \, \mid \, y_j \neq \tilde{y}_j \}$.
A natural metric on $Y$ is defined by
$$
dist(\y, \tilde{\y}) =
\begin{cases}
(1-p)^{k(\y, \tilde{\y})}, & \text{ if } \y \neq \tilde{\y};\\
0, & \text{ if } \y = \tilde{\y}.
\end{cases}
$$
and the induced topology makes $Y$ totally disconnected.
The compactness of $Y$ follows from the standard diagonalization argument.
$Y(x)$ is thus a closed subset of $Y$ and also compact.
On the other hand, based on (\ref{eq:path_recursive_general}),
$J: Y \mapsto R$ is Lipschitz-continuous.
As a continuous function on a compact
set it must attain the minimum at some $\ybar \in Y(x)$.
\end{proof}

We will use $Y^*(x)$ to denote the set of all such {\em minimizing paths}
starting from $x$.

\begin{lemma}
Suppose $\y = (y_0, y_1, \ldots) \in Y^*(x),$
while $k$ and $m$ are non-negative integers such that $k < m$ and
$y_k = y_m$.  Then $J(\y) = J(y_0, \ldots, y_{k}, y_{m+1}, y_{m+2}, \ldots)$.
\label{lm:min_equally_good_tail}
\end{lemma}
\begin{proof}
We show that any loop can be removed from an optimal (not necessarily simple) path without increasing
that path's total cost.
By the optimality of $\y$,
$
J(\y) \leq J(y_0, \ldots, y_{k}, y_{m+1}, y_{m+2}, \ldots).
$
On the other hand, if
$
J(\y) < J(y_0, \ldots y_{k}, y_{m+1}, y_{m+2}, \ldots),
$
then, by formula (\ref{eq:path_recursive_general}),
$
J(y_{k}, y_{k+1}, y_{k+2},\ldots) <
J(y_{k}, y_{m+1}, y_{m+2}, \ldots).
$
The latter contradicts the optimality of $\y$ 
since it implies that
$$
J(\y) = J(y_0, \ldots, y_m, y_{m+1}, y_{m+2}, \ldots) >
J(y_0, \ldots, y_m, y_{k+1}, y_{k+2}, \ldots).
$$
\end{proof}

\begin{lemma}
Suppose $\y = (y_0, y_1, \ldots) \in Y^*(x),$ while
$k$ and $m$ are non-negative integers such that $k < m$ and
$y_k = y_m$.
Define $\tilde{\y}$ by replacing the tail $(y_{k}, \ldots)$
with an infinitely repeated loop $(y_{k}, y_{k+1}, \ldots, y_{m-1})$.
Then $J(\y) = J(\tilde{\y}).$
\label{lm:min_equally_good_loop}
\end{lemma}
\begin{proof}
First, note that proving the equality for $k=0$ combined with the
formula (\ref{eq:path_recursive_general}) yields the proof for the general case.
Assuming $k=0$,
\begin{equation}
J(\tilde{\y}) =
\sum_{r=0}^{\infty}
\left(
\sum_{i=0}^{m-1} \left( K(y_i, y_{i+1}) + p q(y_{i+1}) \right) (1-p)^{i+r m}
\right)
=
\left(
\sum_{i=0}^{m-1} \left( K(y_i, y_{i+1}) + p q(y_{i+1}) \right) (1-p)^{i}
\right)
\frac{1}{1 - (1-p)^m}.
\end{equation}

By Lemma \ref{lm:min_equally_good_tail} and formula (\ref{eq:path_recursive_general}),
$$
J(\y) \; = \; \sum_{i=0}^{m-1} \left( K(y_i, y_{i+1}) + p q(y_{i+1}) \right) (1-p)^i
\; + \;
(1-p)^m \, J(\y);
$$
$$
\sum_{i=0}^{m-1} \left( K(y_i, y_{i+1}) + p q(y_{i+1}) \right) (1-p)^i
\; = \;
\left( 1 - (1-p)^m \right) \, J(\y),
$$
which yields $J(\tilde{\y}) = J(\y).$
\end{proof}

\begin{thm}
There exists a stationary policy $\mu$ such that $\J(x, \mu) = V(x)$.
\label{thm:stationary}
\end{thm}
\begin{proof}
By Lemma \ref{lm:min_attained}, there exists some optimal path
$\yhat \in Y^*(x).$
We note that, given any path, it is easy to define a policy generating it.
Moreover, if the path is {\em simple}, then it can be produced
by a single control mapping (and the corresponding policy will be stationary).
So, we simply need to prove the existence of some
$\ybar \in Y^s(x)$ such that $J(\ybar) = J(\yhat).$\\
For each non-negative integer $k$ and each path $\y \in Y(x)$, let
$M(\y,k) = \{m \, \mid \, m > k, \, y_m = y_k\}$.
We will define a function $D_k: Y(x) \mapsto Y(x)$ as follows:\\
$\bullet \,$ if $M(\y,k)$ is not empty, then
$D_k(\y) = \tilde{\y}$ produced in Lemma \ref{lm:min_equally_good_loop}
using $k$ and $m = \min M(\y,k)$,\\
$\bullet \,$ and $D_k(\y) = \y$ otherwise.

Note, that if $\y \in Y^*(x)$, then $D_k(\y) \in Y^*(x)$ as well.
Defining $\tilde{\y}^0 = \yhat$ and $\tilde{\y}^k = D_{k-1} (\tilde{\y}^{k-1})$
we obtain a simple path in at most $M$ steps; i.e.,
$\ybar = \tilde{\y}^M \in Y^s(x)$.
\end{proof}

\section{Appendix B: a quadrant-by-quadrant update formula.}
\label{s:Appendix_B}
To simplify the discussion, we will focus on one node
$\x = \x_{i,j}$, renaming its neighbors as in Figure \ref{fig:semi-L} and
slightly abusing the notation as follows:
$$
\V = V_{i,j}, \qquad K = K_{i,j}, \qquad q = q_{i,j}, \qquad f = f_{i,j},
$$
$$
V_1 = V_{i+1,j}, \qquad V_2 = V_{i,j+1}, \qquad
V_3 = V_{i-1,j}, \qquad V_4 = V_{i,j-1}.
$$

First, suppose that $\V<q, \;
\max \left( D^{-y}_{ij}V, \, -D^{+y}_{ij}V, \, 0 \right) = 0,$ and
\begin{equation}
\max \left( D^{-x}_{ij}V, \, -D^{+x}_{ij}V, \, 0 \right) = -D^{+x}_{ij}V
\quad \Longrightarrow \quad \V \geq V_1,
\label{eq:use_Vx_plus}
\end{equation}
Then \eqref{eq:DP_cUHP} reduces to
$\V  = q +  \left[K - \frac{f}{h} (\V-V_1) \right] / \lambda,
$
with the solution
\begin{equation}
\V = (h K + \lambda h q + f V_1) / (\lambda h + f).
\label{eq:one_sided_update}
\end{equation}
We note that $\V<q$ implies $\V - V_1 = h (K + \lambda (q - \V)) / f > 0$,
which is consistent with \eqref{eq:use_Vx_plus}.

Now suppose that $\V<q$ and \eqref{eq:use_Vx_plus} hold, but
\begin{equation}
\max \left( D^{-y}_{ij}V, \, -D^{+y}_{ij}V, \, 0 \right) = -D^{+y}_{ij}V,
\quad \Longrightarrow \quad \V \geq V_2,
\label{eq:use_Vy_plus}
\end{equation}
Then \eqref{eq:DP_cUHP} reduces to a quadratic equation
\begin{equation}
f^2 \left[
\left( \frac{\V - V_1}{h} \right)^2
+
\left( \frac{\V - V_2}{h} \right)^2
\right]
=
\left[ K + \lambda q - \lambda \V \right]^2.
\label{eq:cUHP_quadrant}
\end{equation}

Suppose $V^{\#}$ is the smallest number satisfying both
\eqref{eq:cUHP_quadrant}
and $V^{\#} \geq \max(V_1, V_2)$ (for consistency with \eqref{eq:use_Vx_plus}
and \eqref{eq:use_Vy_plus}).
If such $V^{\#}$ exists,
we use it as an ``update'' from this first quadrant; i.e., $V^{12} = V^{\#}$.
Otherwise, we define the update similarly
to \eqref{eq:one_sided_update} as
$V^{12} = (h K + \lambda h q + f \min(V_1, V_2)) / (\lambda h + f).$
The updates from the other quadrants are similarly defined
using $(V_2,V_3)$, $(V_3, V_4)$, and $(V_4, V_1).$
If the values in $NV_{i,j}$ are fixed,
it is easy to check that
\begin{equation}
V_{i,j} = \min \left(q, V^{12}, V^{23}, V^{34}, V^{41} \right)
\label{eq:full_update}
\end{equation}
is the unique solution of \eqref{eq:DP_cUHP}.

\begin{remark}
\label{rem:good_update}
The modified version of the Fast Marching Method in section \ref{ss:FMM_modified}
calls for updating the not-yet-$Accepted$ value of $V_{i,j}$ whenever one of
its neighbors is $Accepted$.  In principle, the above procedure and
formula \eqref{eq:full_update} may be used
for that update, but a more efficient implementation can be built utilizing
the quadrant-by-quadrant approach and the fact that $V_{i,j}$ depends on
the smaller neighbors only.

First, each quadrant-update should
be considered only if at least one of the neighbors defining this quadrant
is already $Accepted$  (by the causality of \eqref{eq:DP_cUHP},
the  non-$Accepted$ values can be replaced by $+\infty$ without affecting $V_{i,j}$).
Second, if $\xBar$ is the last $Accepted$ gridpoint, then only
quadrants in which it participates will need to be considered
(since the updates from other quadrants have already been previously stored in
$V_{i,j}$ -- this is similar to the idea behind the update formula
\eqref{eq:temp_labels_local_Determ} for Dijkstra's method).
Third, the monotonicity of \eqref{eq:DP_cUHP} guarantees that at most one quadrant
is relevant.\\

For definiteness sake, suppose that the recently
accepted $\xBar = \x_{i+1,j} = \x_1$ in the  notation of Figure \ref{fig:semi-L}.
The efficient update procedure would set
$V(\x_{i,j}) \, := \; \min \left( V(\x_{i,j}), \V \right)$,
where $\V$ is\\

\noindent
1. computed by formula \eqref{eq:one_sided_update} if neither $\x_2$ nor $\x_4$ is $Accepted$;\\
2. equal to $V^{12}$ if $\x_2$ is $Accepted$ and $\x_4$ is not (or if both of them are $Accepted$, but $V_2 \leq V_4$);\\
3. equal to $V^{41}$ if $\x_4$ is $Accepted$ and $\x_2$ is not (or if both of them are $Accepted$, but $V_2 > V_4$).
\end{remark}

\subsection{A semi-Lagrangian discretization.}

The upwind finite difference discretization \eqref{eq:DP_cUHP}
and its quadrant-by-quadrant version in the previous section
do not appear natural from the point of view of optimal control.
A more natural semi-Lagrangian scheme is based on
a direct discretization of the optimality principle.
Here we show that the former is in fact equivalent to the latter.

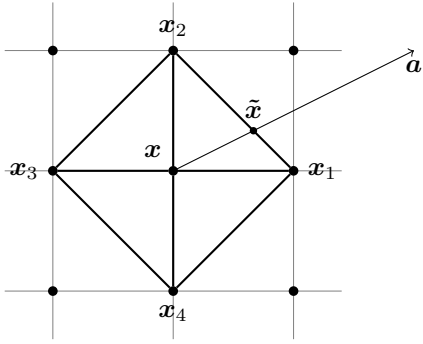
\begin{figure}
\center{
\begin{tikzpicture}[scale=1.6]
\tikzstyle{gridpoint}=[circle,draw=black!100,fill=black!100,thick,inner sep=0pt,minimum size=1mm]
\tikzstyle{interpoint}=[circle,draw=black!100,fill=black!100,thick,inner sep=0pt,minimum size=0.7mm]
\draw[step=1cm,gray,very thin] (-1.4,-1.4) grid (1.4,1.4);

  \node[gridpoint]         (x)		at (0,0)	[label=120:{$\x$}]		{};
  \node[gridpoint]         (x1)		at (1,0)	[label=right:{$\x_1$}]				{};
  \node[gridpoint]         (x2)		at (1,1)	[label=above right:{}]			{};
  \node[gridpoint]         (x3)		at (0,1)	[label=above:{$\x_2$}]				{};
  \node[gridpoint]         (x4)		at (-1,1)	[label=above left:{}]			{};
  \node[gridpoint]         (x5)		at (-1,0)	[label=left:{$\x_3$}]				{};
  \node[gridpoint]         (x6)		at (-1,-1)	[label=below left:{}]			{};
  \node[gridpoint]         (x7)		at (0,-1)	[label=below:{$\x_4$}]				{};
  \node[gridpoint]         (x8)		at (1,-1)	[label=below right:{}]			{};

  \path (x)		edge	[thick]		node {}		(x1)
				edge	[thick]		node {}		(x3)
				edge	[thick]		node {}		(x5)
				edge	[thick]		node {}		(x7);
  \path (x1)	edge	[thick]		node {}		(x3)
				edge	[thick]		node {}		(x7);
  \path (x5)	edge	[thick]		node {}		(x3)
				edge	[thick]		node {}		(x7);

  \draw[->]		(0,0)	--	(2,1)	node[below] {$\ba$};

  \node[interpoint]       (xtilde)	at (0.6666,0.3333)	[label=above:{$\xBtilde$}]	{};
\end{tikzpicture}
}
\caption{
{\footnotesize
A semi-Lagrangian scheme using a simple 5-point stencil on a Cartesian grid.
}
}
\label{fig:semi-L}
\end{figure}

Suppose $\x = (x_i, y_j)$ and $v(\x) < q(\x)$
(i.e., $\ba_0$ is not the optimal control value at $\x$).
Suppose $f$ and $K$ are locally constant, $v$ is smooth, and
the optimal direction of motion from $\x$ is $\ba \in S_1$
lying in the first quadrant; see Figure \ref{fig:semi-L}.
We assume that the motion continues in the direction $\ba$ until crossing
the segment $\x_1\x_2$ at the point $\xBtilde$ after $\tau  = |\x-\xBtilde| / f$ units of time.
The corresponding running cost is $\tau K$ and the probability of
termination signal received while moving from $\x$ to $\xBtilde$ is
$(1 - e^{-\lambda \tau})$.  In case of termination, we incur the terminal cost of
$q$; otherwise, the motion along the approximate optimal trajectory continues
(with the expected cost of $v(\xBtilde)$).
Thus,
\begin{eqnarray*}
v(\x)  &=& K \tau + (1 - e^{-\lambda \tau}) q  +  e^{-\lambda \tau} v(\xBtilde) + O(\tau^2)
= \frac{K \tau e^{\lambda \tau}}{e^{\lambda \tau}} +
\frac{(e^{\lambda \tau} - 1) q}{e^{\lambda \tau}}  +
\frac{v(\xBtilde)}{e^{\lambda \tau}} + O(\tau^2)\\
&=& \frac{K \tau }{1+\lambda \tau} +
\frac{q \lambda \tau}{1+\lambda \tau}  +
\frac{v(\xBtilde)}{1+\lambda \tau} + O(\tau^2).
\end{eqnarray*}
Let
$\Xi = \left\{ \xi = (\xi_1, \xi_2) \, \mid \,
\xi_1 + \xi_2 = 1
\text { and } \forall \xi_1, \xi_2 \geq 0
\right\}.$
If $\xi \in \Xi$ is such that
$\xBtilde = \xi_1 \x_1 + \xi_2 \x_2$, then
$$
\tau(\xi) =  (h / f) \sqrt{\xi_1^2 + \xi_2^2}
\qquad \text{and} \qquad
v(\xBtilde) \approx \xi_1 V_1 + \xi_2 V_2.$$
This suggests a semi-Lagrangian scheme for the grid function $V$:
\begin{equation}
V^{12} = \min\limits_{\xi \in \Xi} C(\xi) = \min\limits_{\xi \in \Xi} \left\{
\frac{(K + \lambda q) \tau(\xi) \; + \; (\xi_1 V_1 + \xi_2 V_2)}{1+\lambda \tau(\xi)}
\right\}.
\label{eq:semiL}
\end{equation}
Given similarly defined updates from all other quadrants,
we can again set
$$
\V = \min \left(q, V^{12}, V^{23}, V^{34}, V^{41} \right).
$$

Kuhn-Tucker optimality conditions can be used to relate this scheme
to \eqref{eq:DP_cUHP}.  A similar connection was previously
demonstrated for Eikonal PDEs on a uniform Cartesian grid by Tsitsiklis \cite{Tsitsiklis},
and then on triangulated meshes and for more general
Hamilton-Jacobi-Bellman PDEs by Sethian and Vladimirsky \cite[Appendix]{SethVlad3};
see also the detailed discussion
of connections to MSSP problems in \cite{VladMSSP}.
The following proof applies the same ideas to the variational inequality \eqref{eq:obstacle_form_iso},
with additional technical details due to the direct dependence of the Hamiltonian on $v$.
\begin{thm}
Let $\xi^*$ denote the minimizer in \eqref{eq:semiL}, and
suppose that $\V = V^{12} < q. \quad$
Then\\
1.  $\quad \xi_i^* > 0  \quad \imply \quad V^{12} > V_i \qquad \text{for } i=1,2.$\\
2.  $\quad \xi_1^*,\xi_2^*  > 0  \quad \imply \quad V^{12}$ defined by \eqref{eq:semiL}
is also a solution of \eqref{eq:cUHP_quadrant}.
\end{thm}
\begin{proof}
We note three useful properties of the function $\tau(\xi)$:
\begin{eqnarray}
\tauP{i}(\xi) &=&
\frac{h^2}{f^2} \frac{\xi_i}{\tau(\xi)};
\label{eq:tau_prop1}
\\
\tau(\xi) &=& \xi_1 \tauP{1}(\xi) + \xi_2 \tauP{2}(\xi);
\label{eq:tau_prop2}\\
\frac{h^2}{f^2} &=& \left(\tauP{1}(\xi) \right)^2 + \left(\tauP{2}(\xi) \right)^2.
\label{eq:tau_prop3}
\end{eqnarray}
To simplify the notation, we will suppress the arguments -- in what follows,
$\tau$ and its partial derivatives are always evaluated at $\xi^*$.

First, note that if $\xi^* = (1,0)$, then $\tau = h/f$ and \eqref{eq:semiL}
reduces to \eqref{eq:one_sided_update};
a similar formula (with $V_2$ replacing $V_1$) holds when $\xi^* = (0,1)$.
Now suppose $\xi_1^*, \xi_2^* > 0.$
By the Kuhn-Tucker optimality conditions applied to $C(\xi)$
defined in formula \eqref{eq:semiL}, there exists a Lagrange multiplier $\mu$
such that
\begin{equation}
\mu = \frac{\partial C }{\partial \xi_i} (\xi^*)
= \frac{(K + \lambda q) \tauP{i} + V_i}{1+ \lambda \tau}
- \frac{\lambda \tauP{i} [\tau (K + \lambda q) + \xi_1^* V_1 + \xi_2^* V_2] }
{(1+ \lambda \tau)^2}
= \frac{(K + \lambda q - \lambda V^{12}) \tauP{i} + V_i}{1+ \lambda \tau},
\qquad \text{for } i=1,2.
\label{eq:L_Mult1}
\end{equation}
Multiplying the above by $\xi_i^*$,
adding up (for $i=1,2$) and using \eqref{eq:tau_prop2}, we see that
\begin{equation}
\mu = \mu (\xi_1 + \xi_2) =
\frac{(K + \lambda q  - \lambda V^{12}) \tau +
\xi_1^* V_1 + \xi_2^* V_2}{1+ \lambda \tau}
= V^{12} - \frac{\lambda \tau V^{12}}{1+ \lambda \tau}
= \frac{V^{12}}{1+ \lambda \tau}.
\label{eq:L_Mult2}
\end{equation}
Combining \eqref{eq:L_Mult1} and \eqref{eq:L_Mult2}, we obtain
\begin{equation}
V^{12} - V_i \; = \; (K + \lambda q - \lambda V^{12}) \tauP{i} \; > \; 0,
\qquad \text{for } i=1,2
\label{eq:L_Mult3}
\end{equation}
where the inequality follows from $K \geq 0$, $V^{12} < q$,
and the fact that $(\xi_i^* > 0  \quad \imply \quad \tauP{i}(\xi^*) > 0)$
by formula \eqref{eq:tau_prop1}.  This shows the causality of
the semi-Lagrangian discretization.

To prove the second half of the theorem,
we take a square of both sides of \eqref{eq:L_Mult3}, sum over $i=1,2$ and
use \eqref{eq:tau_prop3} to obtain
$$
\left( V^{12} - V_1 \right)^2 + \left( V^{12} - V_2 \right)^2 =
\frac{h^2}{f^2} \left(K + \lambda q - \lambda V^{12} \right)^2,
$$
which is equivalent to \eqref{eq:cUHP_quadrant}.
\end{proof}

\end{document}